\documentclass[12pt,reqno]{amsart}
\usepackage{amssymb}
\usepackage[latin1]{inputenc}
\usepackage{dsfont}
\usepackage{hyperref}
\usepackage{fullpage}

\newtheorem{theorem}{Theorem}[section]

\newtheorem{proposition}[theorem]{Proposition}
\newtheorem{corollary}[theorem]{Corollary}
\theoremstyle{definition}
\newtheorem{definition}[theorem]{Definition}

\theoremstyle{remark}
\newtheorem{remark}[theorem]{Remark}

\numberwithin{equation}{section}

\begin{document}

\title[On $\sigma$-finite measures related to the Martin boundary of recurrent Markov chains]
{On $\sigma$-finite measures related to the Martin boundary of recurrent Markov chains}
\author[J. Najnudel]{Joseph Najnudel}
\address{Institut de Math\'ematique de Toulouse, Universit\'e Paul Sabatier, Toulouse}
\email{\href{mailto:joseph.najnudel@math.univ-toulouse.fr}{joseph.najnudel@math.univ-toulouse.fr}}

\date{\today}

\begin{abstract}
In our monograph with Roynette and Yor \cite{NRY}, we construct a $\sigma$-finite measure related to 
penalisations of different stochastic processes, including the Brownian motion in dimension 1 or 2, and 
a large class of linear diffusions. In the 
last chapter of the monograph, we define similar measures from 
recurrent Markov chains satisfying some technical conditions. In the present paper, 
we give a classification of these measures, in function of the minimal Martin boundary of the Markov chain
considered at the beginning. We apply 
this classification to the examples 
considered at the end of \cite{NRY}.
\end{abstract} 

\maketitle

\section{Introduction}
In a number of articles by Roynette, Vallois and Yor, summarized in \cite{RVY}, 
the authors study many examples of probability measures on the space of continuous functions from 
$\mathbb{R}_+$ to $\mathbb{R}$,
which are obtained as weak limits of absolutely continuous measures, with respect to the law of 
the Brownian motion. 
More precisely, one considers  the Wiener measure $\mathbb{W}$
on the space $\mathcal{C}(\mathbb{R}_+, \mathbb{R})$ of
continuous functions from $\mathbb{R}_+$ to $\mathbb{R}$, 
and endowed with its canonical filtration 
$(\mathcal{F}_s)_{s \geq 0}$, and the following $\sigma$-algebra  $$\mathcal{F}:= 
\underset{s \geq 0}{\bigvee} \mathcal{F}_s.$$ 
One then considers
 $(\Gamma_t)_{t \geq 0}$, a family of 
nonnegative random variables on $\mathcal{C}(\mathbb{R}_+, \mathbb{R})$, such that 
$$0 < \mathbb{E}_{\mathbb{W}} [\Gamma_t]  < \infty,$$
and for $t \geq 0$, one defines the probability measure $$\mathbb{Q}_t := 
\frac{\Gamma_t}{\mathbb{E}_{\mathbb{W}}[\Gamma_t]} \, . \mathbb{W}.$$
 Under these assumptions, Roynette, Vallois and Yor have 
shown that for many examples of families of functionals $(\Gamma_t)_{t \geq 0}$, one can find a probability 
measure $\mathbb{Q}_{\infty}$ satisfying the following property: for all $s \geq 0$ and for all events 
$\Lambda_s \in \mathcal{F}_s$,
$$\mathbb{Q}_t [\Lambda_s] \underset{t \rightarrow \infty}{\longrightarrow} \mathbb{Q}_{\infty} [\Lambda_s].$$

In our monograph with Roynette and Yor \cite{NRY}, Chapter 1, we show that for a large class of 
functionals $(\Gamma_t)_{t \geq 0}$, the measure $\mathbb{Q}_{\infty}$ exists and is absolutely continuous 
with respect to a $\sigma$-finite measure $\mathbf{W}$, which is explicitly described and which satisfies some remarkable 
properties. In Chapters 2, 3 and 4 of the monograph, we construct an analog of the measure $\mathbf{W}$, respectively 
for the two-dimensional Brownian motion, for a large class of linear diffusions, and for a large class of 
recurrent Markov chains. In a series of papers with Nikeghbali (see \cite{NN1} and \cite{NN2}), we generalize the 
construction to submartingales $(X_s)_{s \geq 0}$ satisfying some technical conditions we do not detail here, 
and such that $X_s = N_s + A_s$, where 
$(N_s)_{s \geq 0}$ is a c\`adl\`ag martingale, $(A_s)_{s \geq 0}$ is an increasing process, and the measure $(dA_s)$ is 
carried by the set $\{s \geq 0, X_s = 0\}$. This class of submartingales, called $(\Sigma)$, was first
 introduced by Yor in \cite{Y}, and their main properties have been studied in detail by Nikeghbali in \cite{N}. 
 
 In the present paper, we focus on the setting of the recurrent Markov chains, stated in Chapter 4 of \cite{NRY}. Our main goal is to classify 
 the $\sigma$-finite measures which 
 can be obtained by the construction 
 given in the monograph.  
 In Section \ref{setting}, we summarize the most important ideas of this construction, 
 and we state some of the main properties of the corresponding $\sigma$-finite measures.  
 In Section \ref{martin}, we show that these measures can be classified via the 
 theory of Martin boundary, adapted to the case of recurrent Markov chains. 
 In Section \ref{Qalpha}, we study the behavior of the canonical trajectory under 
 some particular measures deduced from the classification given in Section \ref{martin}.  
 In Section \ref{examples}, we apply our results to the examples considered at the 
 end of our monograph \cite{NRY}. 
 
\section{The main setting} \label{setting}
Let $E$ be a countable set, $(X_n)_{n \geq 0}$ the canonical process on $E^{\mathbb{N}_0}$, 
$(\mathcal{F}_n)_{n \geq 0}$ its natural filtration, and $\mathcal{F}_{\infty}$ the $\sigma$-algebra 
generated by $(X_n)_{n \geq 0}$. We define $(\mathbb{P}_x)_{x \in E}$ as a
family of probability measures on the filtered measurable space $(E^{\mathbb{N}_0}, (\mathcal{F}_n)_{n \geq
0}, \mathcal{F}_{\infty})$ which corresponds to a Markov chain, i.e. there exists 
a family $(p_{y,z})_{y,z \in E}$ of elements in $[0,1]$ such that for all $k \geq 0$, $x_0, \dots, 
x_k \in E$, 
$$\mathbb{P}_x (X_0 = x_0, X_1 = x_1, \dots, X_k = x_k) = \mathds{1}_{x_0 = x} 
p_{x_0, x_1} p_{x_1, x_2} \dots p_{x_{k-1}, x_k}.$$
The expectation under $\mathbb{P}_x$ will be denoted $\mathbb{E}_x$. Moreover, we assume the following properties: 
\begin{itemize}
 \item For all $x \in E$, $p_{x,y} = 0$ for all but finitely many $y \in E$. 
 \item The Markov chain is irreducible, i.e. for all $x, y \in E$, there exists $n \geq 0$ such that 
 $\mathbb{P}_{x} (X_n = y) > 0$. 
 \item The Markov chain is recurrent, i.e. for all $x \in E$, $\mathbb{P}_x \left( \sum_{n \geq 0} 
 \mathds{1}_{X_n = x} = \infty \right) = 1$. 
\end{itemize}
 
Using the results in Chapter 4 of \cite{NRY}, the following proposition is not difficult to prove: 
 \begin{proposition} \label{NRY4}
  Let $x_0 \in E$, and let $\varphi$ be a function from $E$ to $\mathbb{R}_+$, such that 
  $\varphi(x_0) = 0$, and $\varphi$ is harmonic everywhere except at $x_0$, i.e. for all 
  $x \neq x_0$, 
  $$\mathbb{E}_x [\varphi(X_1)] = \sum_{y \in E} p_{x,y} \varphi(y) = \varphi(x).$$
  Then, there exists a family of $\sigma$-finite measures $(\mathbb{Q}^{x_0, \varphi}_x)_{x \in E}$ on 
  $(E^{\mathbb{N}_0}, \mathcal{F}_{\infty})$ satisfying the following properties: 
  \begin{itemize}
  \item For all $x \in E$, the canonical process starts at $x$ under $\mathbb{Q}^{x_0, \varphi}_x$, i.e.
  $$\mathbb{Q}^{x_0, \varphi}_x \left( X_0 \neq x \right) = 0.$$
   \item For all $x \in E$, the canonical process is transient under $\mathbb{Q}^{x_0, \varphi}_x$, 
   i.e. for all $x, y \in E$, 
   $$\mathbb{Q}^{x_0, \varphi}_x \left(\sum_{n \geq 1} \mathds{1}_{X_n = y} =
   \infty  \right) = 0.$$
   \item For all $n \geq 0$, for all nonnegative, $\mathcal{F}_n$-measurable functionals $F_n$, 
   $$\mathbb{Q}^{x_0, \varphi}_x \left( F_n \, \mathds{1}_{\forall k \geq n, X_k \neq x_0} \right)
   = \mathbb{E}_{x} [F_n \varphi(X_n)],$$
   where $\mathbb{Q}^{x_0, \varphi}_x (H)$ denotes the integral of $H$ with respect to $\mathbb{Q}^{x_0, \varphi}_x$.
  \end{itemize}
  Moreover, the two last items are sufficient to determine uniquely the measure $\mathbb{Q}^{x_0, \varphi}_x$.
 \end{proposition}
 \begin{remark}
  If we refer to our joint work with Nikeghbali \cite{NN3}, we observe that under $\mathbb{P}_x$, 
  $(\varphi(X_n))_{n \geq 0}$ is a the discrete-time submartingale of class $(\Sigma)$, as stated in 
  Theorem 3.5 of \cite{NN3}. Moreover, one checks that $\mathbb{Q}^{x_0, \varphi}_x$ is the corresponding
  $\sigma$-finite measure, denoted $\mathcal{Q}$ in \cite{NN3}, as soon as 
  $\varphi(y) > 0 $ for all $y \neq x_0$. 
 \end{remark}

 \begin{proof}
  The second and third items are respectively Proposition 4.2.3 and Corollary 4.2.6 of \cite{NRY}. 
  The first item is a consequence of the second and the third. Indeed, by taking $F_n = \mathds{1}_{X_0 \neq x}$, 
  we get for all $n \geq 0$, 
  $$\mathbb{Q}^{x_0, \varphi}_x \left(\mathds{1}_{X_0 \neq x, \forall k \geq n, X_k \neq x_0} \right)
  = \mathbb{E}_{x} [ \varphi(X_n) \mathds{1}_{X_0 \neq x} ] = 0,$$
  and then, by taking the union for all $n \geq 0$,
  $$ \mathbb{Q}^{x_0, \varphi}_x  \left(X_0 \neq x, \exists n
  \geq 0, \forall k \geq n, X_k \neq x_0 \right) = \mathbb{Q}^{x_0, \varphi}_x
 \left(X_0 \neq x, \sum_{k \geq 0} \mathds{1}_{X_k = x_0} < \infty \right) = 0,
  $$
and then 
$$ \mathbb{Q}^{x_0, \varphi}_x (X_0 \neq x) = 0$$
by the transience of the canonical process under $\mathbb{Q}^{x_0, \varphi}_x$. 

It remains to prove that the second and the third items uniquely determine 
$\mathbb{Q}^{x_0, \varphi}_x$. From the third item, we have for all $n \geq m \geq 0$, 
\begin{align*}\mathbb{Q}^{x_0, \varphi}_x \left( F_n \mathds{1}_{\forall k \geq m, X_k \neq x_0} \right)
& = \mathbb{Q}^{x_0, \varphi}_x \left( F_n \mathds{1}_{ \forall k \in \{m, m+1, \dots, n-1\}, 
X_k \neq x_0} \mathds{1}_{\forall k \geq n, X_k \neq x_0} \right)
\\ & =  \mathbb{E}_x[ F_n \mathds{1}_{ \forall k \in \{m, m+1, \dots, n-1\}, 
X_k \neq x_0} \varphi(X_n) ].
\end{align*}
Moreover, 
$$ \mathbb{Q}^{x_0, \varphi}_x \left( \forall k \geq m, X_k \neq x_0 \right) 
= \mathbb{E}_x [ \varphi(X_m)] < \infty,$$
since $\varphi(X_m)$ is almost surely in a finite subset of $E$. Indeed, by assumption, 
for all $y \in E$, there exist only finitely many $z \in E$ such that $p_{y,z} > 0$. 
Hence, the measure 
$$ \mathds{1}_{\forall k \geq m, X_k \neq x_0}  \cdot  \mathbb{Q}^{x_0, \varphi}_x $$ 
is finite and uniquely determined for all sets in $\mathcal{F}_n$, for all $n \geq 0$. 
By monotone class theorem, this measure is uniquely determined. Taking the increasing limit 
for $m \rightarrow \infty$, the measure 
$$ \mathds{1}_{\exists m \geq 0, \forall k \geq m, X_k \neq x_0}  \cdot  \mathbb{Q}^{x_0, \varphi}_x$$
is also uniquely determined. Now, the property of transience which is assumed implies that this measure 
is $ \mathbb{Q}^{x_0, \varphi}_x$. 
 \end{proof}
 The homogeneity of the Markov chain can be stated as follows: for all $n \geq 1$, 
 $x, y \in E$, 
 $$\mathds{1}_{X_n = y} \cdot \mathbb{P}_x = \mathds{1}_{X_n = y} \cdot (\mathbb{P}^{(n)}_{x} 
 \circ  \mathbb{P}_y),$$
 where $\mathbb{P}^{(n)} \circ \mathbb{Q}$ denotes the image of $\mathbb{P} \otimes \mathbb{Q}$ by 
 the map from $E^{\mathbb{N}_0} \times E^{\mathbb{N}_0}$ to $E^{\mathbb{N}_0}$, given by 
 $$((x_0, x_1, x_2, \dots),(y_0,y_1, y_2 \dots)) \mapsto (x_0, x_1,x_2,\dots, x_n, y_1, y_2, y_3, \dots).$$
 In other words, conditionally on $X_n = y$, 
 the canonical trajectory under $\mathbb{P}_x$ has the same law as the concatenation of 
 the $n$ first steps of the canonical trajectory under $\mathbb{P}_x$, and an independant trajectory 
 following $\mathbb{P}_y$. 
 
 The following result shows that the family of
 measures $(\mathbb{Q}_x^{x_0, \varphi})_{x \in E}$ satisfies a similar property: informally, it can be 
 obtained  from $(\mathbb{P}_x)_{x \in E}$ and 
 $(\mathbb{Q}_x^{x_0, \varphi})_{x \in E}$ itself by concatenation of the trajectories. 
 \begin{proposition}
  For all $n \geq 0$, $x, y \in E$, 
  $$ \mathds{1}_{X_n = y} \cdot \mathbb{Q}_x^{x_0, \varphi}= \mathds{1}_{X_n = y}
  \cdot (\mathbb{P}^{(n)}_{x} 
 \circ  \mathbb{Q}_y^{x_0, \varphi}).$$
 \end{proposition}
\begin{proof}
If $n = 0$, the two sides of the equality 
vanish for all $y \neq x$, since 
the canonical trajectory starts at $x$ 
under $\mathbb{Q}_x^{x_0, \varphi}$ and 
$\mathbb{P}_x$. If $n = 0$ and 
$y = x$, the equality we want to show is also immediate. Hence,  we can assume $n \geq 1$. 
 Let $p \geq n \geq 1$, and let $F_p$ be a $\mathcal{F}_p$-measurable, nonnegative functional.
 We have 
 $$ \left(\mathds{1}_{X_n = y} \cdot \mathbb{Q}_x^{x_0, \varphi} \right) 
\left(F_p \mathds{1}_{\forall k \geq p, X_k \neq x_0} \right) 
 = \mathbb{Q}_x^{x_0, \varphi} \left(F_p \mathds{1}_{X_n = y} \mathds{1}_{\forall k \geq p, X_k \neq x_0} 
\right)
 = \mathbb{E}_x \left[ F_p \mathds{1}_{X_n = y} \varphi(X_p) \right]$$ $$
= \sum_{x_1, x_2, \dots, x_{n-1} \in E} p_{x,x_1} p_{x_1,x_2} \dots, p_{x_{n-1}, y} 
\mathbb{E}_x[ F_p \varphi(X_p) | X_1 = x_1, \dots, X_{n-1} = x_{n-1}, X_n = y].$$
Now, the functional $F_p$ is nonnegative and $\mathcal{F}_p$-measurable, so there exists a function $\Phi$
from $E^{p+1}$ to $\mathbb{R}_+$ such that 
$$F_p = \Phi(X_0, \dots, X_p).$$
We get, using the Markov property: 
$$ \left(\mathds{1}_{X_n = y} \cdot \mathbb{Q}_x^{x_0, \varphi} \right) 
\left(F_p \mathds{1}_{\forall k \geq p, X_k \neq x_0} \right) $$ $$ 
= \sum_{x_1, x_2, \dots, x_{n-1} \in E} p_{x,x_1} p_{x_1,x_2} \dots, p_{x_{n-1}, y} 
\mathbb{E}_x[ \Phi(X_0, \dots, X_p) \varphi(X_p) | X_1 = x_1, \dots, X_{n-1} = x_{n-1}, X_n = y]$$
$$  = \sum_{x_1, x_2, \dots, x_{n-1} \in E} p_{x,x_1} p_{x_1,x_2} \dots, p_{x_{n-1}, y} 
\mathbb{E}_y[ \Phi(x_0, x_1, \dots, x_{n-1}, y, X_1, \dots, X_{p - n}) \varphi(X_{p-n})]$$
$$  = \sum_{x_1, x_2, \dots, x_{n-1} \in E} p_{x,x_1} p_{x_1,x_2} \dots, p_{x_{n-1}, y} 
\mathbb{Q}_y^{y_0, \varphi} [ \Phi(x_0, x_1, \dots, x_{n-1}, y, X_1, \dots, X_{p - n}) 
\mathds{1}_{\forall k \geq p-n, X_k \neq x_0} ]
)$$ $$ 
= \mathbb{P}_{x} \otimes \mathbb{Q}_y^{y_0, \varphi} [ \mathds{1}_{Y_n = y} 
\Phi(Y_0, \dots, Y_n, Z_1, Z_2, \dots, Z_{n-p}) 
\mathds{1}_{\forall k \geq p-n, Z_k \neq x_0}  ],$$
where in the last line, $(Y_n, Z_n)_{n \geq 0}$ denotes the canonical process on 
the space $E^{\mathbb{N}_0} \times E^{\mathbb{N}_0}$ where the measure 
$\mathbb{P}_{x} \otimes \mathbb{Q}_y^{x_0, \varphi}$ is defined. 
Hence 
$$ \left(\mathds{1}_{X_n = y} \cdot \mathbb{Q}_x^{x_0, \varphi} \right) 
\left(F_p \mathds{1}_{\forall k \geq p, X_k \neq x_0} \right) $$ 
$$ = \mathbb{P}^{(n)}_{x} \circ \mathbb{Q}_y^{y_0, \varphi}[ \mathds{1}_{X_n = y} 
\Phi(X_0, \dots, X_n, X_{n+1}, X_{n+2}, \dots, X_{p}) 
\mathds{1}_{\forall k \geq p, X_k \neq x_0}  ]$$ 
$$ =( \mathds{1}_{X_n = y}
  \cdot (\mathbb{P}^{(n)}_{x} 
 \circ  \mathbb{Q}_y^{x_0, \varphi}) )\left(F_p \mathds{1}_{\forall k \geq p, X_k \neq x_0} \right).$$
Hence, the two measures stated in the proposition coincide on all functionals of the form 
$F_p \mathds{1}_{\forall k \geq p, X_k \neq x_0}$ if $p \geq n$, and then on all functionals of the 
form 
$$ F_q \mathds{1}_{\forall k \geq p, X_k \neq x_0} = 
(F_q \mathds{1}_{\forall k \in \{p, \dots q-1\}, X_k \neq x_0} )
\mathds{1}_{\forall k \geq q, X_k \neq x_0},$$
for all $q \geq p \geq n$. 
Using the finiteness of the measure $\mathbb{Q}_x^{x_0, \varphi}$ restricted to the set
$\{\forall k \geq p, X_k \neq x_0\}$, and the monotone class theorem, one deduces that 
the two measures we are comparing coincide on all sets included in $\{\forall k \geq p, X_k \neq x_0\}$. 
Taking the union for $p \geq n$, and using the property of transience satisfied by the measures 
 $\mathbb{Q}_y^{x_0, \varphi}$, $y \in E$, one deduces that the two measures we are comparing are equal. 
\end{proof}

Knowing the result we have just proven, it is natural to ask which families of $\sigma$-finite mesures
$(\mathbb{Q}_x)_{x \in E}$ 
 on $(E^{\mathbb{N}_0}, \mathcal{F}_{\infty})$ satisfy, for all $x, y \in E$,
 $n \geq 0$, 
 \begin{equation}
 \mathds{1}_{X_n = y} \cdot \mathbb{Q}_x = \mathds{1}_{X_n = y} \cdot (\mathbb{P}^{(n)}_{x} 
 \circ  \mathbb{Q}_y). \label{1}
 \end{equation}
 We know that all linear combinations, with nonnegative coefficients, 
 of families of the form $(\mathbb{P}_x)_{x \in E}$ and 
 $(\mathbb{Q}^{x_0, \varphi}_x)_{x \in E}$ satisfy this condition. It is natural to ask if there are other 
 such families of measures. We do not know the complete answer. However, we have the following partial result: 
 \begin{proposition} \label{2.4}
  Let $(\mathbb{Q}_x)_{x \in E}$ be a family of $\sigma$-finite measures such that \eqref{1} holds for all 
  $x, y \in E$, $n \geq 0$. Then: 
  \begin{itemize}
   \item If $\mathbb{Q}_x$ is a finite measure for at least one $x \in E$, then there exists $c \geq 0$ such that
   $\mathbb{Q}_x = c \mathbb{P}_x$ for all $x \in E$. 
   \item If 
   for some $x_0, x_1 \in E$, 
   \begin{equation}
  \varphi(x) :=  \mathbb{Q}_x \left( \forall k \geq 0, X_k \neq x_0 \right) < \infty \label{2}
   \end{equation} 
    for all $x \in E$, and 
  \begin{equation}
  \mathbb{Q}_{x_1} \left( \sum_{k=0}^{\infty} \mathds{1}_{X_k = x_0} = \infty \right) 
  = 0, \label{3}
  \end{equation}
   then $\varphi(x_0) = 0$, 
   $\mathbb{E}_x (\varphi(X_1)) = \varphi(x)$ for $x \neq x_0$ and 
   $\mathbb{Q}_x  = \mathbb{Q}^{x_0, \varphi}_x$ for all $x \in E$. 
   \item Moreover, if the conditions of the previous item are satisfied, then 
  \eqref{2} and \eqref{3} are satisfied
  for all $x_0, x_1, x \in E$. The fact that \eqref{3} holds for all $x_0, x_1 \in E$ means that the canonical process is transient under $\mathbb{Q}_{x}$ for all 
  $x \in E$. 
   
  \end{itemize}
 \end{proposition}
\begin{proof}
 Let us assume that $\mathbb{Q}_{x_0}$ has finite total mass for some $x_0 \in E$. 
  For $y \in E$, let $\psi(y)$ be the total mass of $\mathbb{Q}_y$. 
For all $n \geq 1$, for any nonnegative, $\mathcal{F}_n$-measurable functional $F_n$, and for all $y \in E$, we deduce 
from \eqref{1}: 
$$ \mathbb{Q}_{x_0} [ F_n \mathds{1}_{X_n = y}] = (\mathbb{P}^{(n)}_{x_0} 
 \circ  \mathbb{Q}_y)[F_n \mathds{1}_{X_n = y} ] = \mathbb{E}_{x_0}[ F_n  \mathds{1}_{X_n = y}] \mathbb{Q}_y(1) 
 = \psi(y) \mathbb{E}_{x_0}[ F_n  \mathds{1}_{X_n = y}],$$
 and then, by adding these expressions for all $y \in E$: 
 $$\mathbb{Q}_{x_0} [F_n] = \mathbb{E}_{x_0} [F_n \psi(X_n)].$$
 For $F_n = 1$, we get
 $$ \mathbb{E}_{x_0} [\psi(X_n)]  = \psi(x_0).$$
 By assumption, $\psi(x_0) < \infty$, and one deduces that $\psi(X_n)$ is 
 $\mathbb{P}_{x_0}$-almost surely finite for all $n \geq 1$. Since the Markov chain is assumed to be irreducible, one deduces that $\psi(y) < \infty$ for all $y \in E$. 
 On the other hand, if $F_n$ is nonnegative, $\mathcal{F}_n$-measurable, and then also $\mathcal{F}_{n+1}$-measurable, 
 one has 
 $$ \mathbb{E}_{x_0} [F_n \psi(X_{n+1})] = \mathbb{Q}_{x_0} [F_n] = \mathbb{E}_{x_0} [F_n \psi(X_n)],$$
 and then 
 $$\mathbb{E}_{x_0}[\psi(X_{n+1}) | \mathcal{F}_n] = \psi(X_n),$$
 i.e. $(\psi(X_n))_{n \geq 1}$ is a $\mathbb{P}_{x_0}$-martingale. 
 Since this martingale is nonnegative, it converges almost surely. On the other hand, since $(X_n)_{n \geq 0}$ is 
 recurrent and irreducible, it visites all the states infinitely often. One easily deduces that 
 $\psi$ is a constant function, let $c \geq 0$ be this constant. One has 
 $$ \mathbb{Q}_{x_0} [F_n] = c \mathbb{E}_{x_0} [F_n],$$
 and then $\mathbb{Q}_{x_0}$ and $c \mathbb{P}_{x_0}$ are two finite measures which coincide on all $\mathcal{F}_n$, $n \geq 1$, 
 and then on all $\mathcal{F}_{\infty}$. 
 Moreover, since we have now proven that the total mass of $\mathbb{Q}_x$ is finite for all $x \in E$, 
 we can replace, in the previous discussion, $x_0$ by any $x \in E$. We then deduce that 
 $\mathbb{Q}_x = c \mathbb{P}_x$ for all $x \in E$. 
 
 Now, let us assume \eqref{2} and \eqref{3} for some $x_0, x_1 \in E$ and all $x \in E$. 
 For all $y \neq x$, 
 $$\mathds{1}_{X_0 = y}  \cdot \mathbb{Q}_x = \mathds{1}_{X_0 = y} \cdot 
 (\mathbb{P}_x^{(0)} \circ \mathbb{Q}_y)$$
 is the measure identically equal to zero, since $X_0 = x \neq y$ almost everywhere under $\mathbb{P}_x$, and then under $(\mathbb{P}_x^{(0)} \circ \mathbb{Q}_y)$. 
 Hence, $X_0 = x$ almost everywhere under $\mathbb{Q}_x$. It is then obvious that 
 $$\varphi(x_0) = \mathbb{Q}_{x_0}
 (\forall k \geq 0, X_k \neq x_0) = 0.$$
 Moreover, for all $x \neq x_0$, 
 \begin{align*}
 \varphi(x) = \mathbb{Q}_{x}
 (\forall k \geq 0, X_k \neq x_0) 
 & = \mathbb{Q}_{x}
 (\forall k \geq 1, X_k \neq x_0)
 \\ & = \sum_{y \in E} 
  \mathbb{Q}_{x}
 (X_1  = y, \forall k \geq 1, X_k \neq x_0) 
 \\ &= \sum_{y \in E} 
 (\mathbb{P}_x^{(1)} \circ \mathbb{Q}_y)
 (X_1  = y, \forall k \geq 1, X_k \neq x_0)
 \\ & = \sum_{y \in E}
 \mathbb{P}_x [X_1 = y] 
  \mathbb{Q}_{y}
 (\forall k \geq 0, X_k \neq x_0) 
 \\ & = \sum_{y \in E}
 \mathbb{P}_x [X_1 = y] \varphi (y) 
 = \mathbb{E}_{x} [\varphi(X_1)].  
 \end{align*}
 Now, for all $\mathcal{F}_n$-measurable, nonnegative functional $F_n$, we get: 
 \begin{align*}
 \mathbb{Q}_x [F_n \mathds{1}_{\forall 
 k \geq n, X_k \neq x_0}] 
 & = \sum_{y \in E} 
 (\mathbb{P}_x^{(n)} \circ \mathbb{Q}_y)
 [F_n \mathds{1}_{X_n = y, \forall 
 k \geq n, X_k \neq x_0} ]
  \\ & =  \sum_{y \in E} 
  \mathbb{E}_x [F_n \mathds{1}_{X_n = y} 
  ] \mathbb{Q}_y (\forall 
 k \geq 0, X_k \neq x_0)
 \\ & = \sum_{y \in E} \varphi(y) 
  \mathbb{E}_x [F_n \mathds{1}_{X_n = y} 
  ]
  \\ & = \mathbb{E}_x 
  [F_n \varphi(X_n)] = \mathbb{Q}_{x}^{x_0, \varphi} [F_n \mathds{1}_{\forall 
 k \geq n, X_k \neq x_0}].
 \end{align*}
 Hence, $\mathbb{Q}_x$ and $\mathbb{Q}_{x}^{x_0, \varphi}$ coincide for functionals 
of the form $ F_n \mathds{1}_{\forall 
 k \geq n, X_k \neq x_0}$, then 
for functionals of the form 
$$ F_p \mathds{1}_{\forall 
 k \geq n, X_k \neq x_0}
 = F_p \mathds{1}_{\forall k \in \{n, \dots, p-1\}, X_k \neq x_0}\mathds{1}_{\forall k \geq p, X_k \neq x_0}$$
 for $p \geq n$, then for all events
 of the form 
 $ A \cap \{\forall k \geq n, X_k \neq x_0\}$, $A \in \mathcal{F}_{\infty}$, then for their union for $n \geq 1$, i.e. 
 $A \cap \{\exists n \geq 1, 
 \forall k \geq n, X_k \neq x_0\}$. This implies 
 $\mathbb{Q}_x = \mathbb{Q}^{x_0, \varphi}_x$, provided that we check that under 
 these two measures, the canonical process 
 hits $x_0$ finitely many times, i.e. 
 $$ \mathbb{Q}_x  \left( \sum_{k=0}^{\infty} \mathds{1}_{X_k = x_0} = \infty \right) 
  =  \mathbb{Q}^{x_0, \varphi}_x  \left( \sum_{k=0}^{\infty} \mathds{1}_{X_k = x_0} = \infty \right) = 0.$$
 Now, for all $n \geq 0$, we have by 
 assumption  
 \begin{align*}
 0 = \mathbb{Q}_{x_1}  \left( \sum_{k=0}^{\infty} \mathds{1}_{X_k = x_0} = \infty \right) 
 & = \sum_{y \in E} \mathbb{Q}_{x_1}  \left(
 X_n = y,  \sum_{k=n}^{\infty} \mathds{1}_{X_k = x_0} = \infty \right) 
 \\ & = \sum_{y \in E}(\mathbb{P}^{(n)}_{x_1}
 \circ \mathbb{Q}_y) \left(
 X_n = y,  \sum_{k=n}^{\infty} \mathds{1}_{X_k = x_0} = \infty \right) 
 \\ & =  \sum_{y \in E}
 \mathbb{P}_{x_1} (X_n = y) 
 \mathbb{Q}_y \left(\sum_{k=0}^{\infty} \mathds{1}_{X_k = x_0} = \infty \right), 
\end{align*}
which implies that 
$$\mathbb{Q}_y \left(\sum_{k=n}^{\infty} \mathds{1}_{X_k = x_0} = \infty \right) = 0$$
for all $y \in E$ such that $\mathbb{P}_{x_1} (X_n = y) > 0$. 
Since the Markov chain is irreducible, 
we deduce that 
$$ \mathbb{Q}_x \left(\sum_{k=n}^{\infty} \mathds{1}_{X_k = x_0} = \infty \right) = 0$$
for all $x \in E$. 
On the other hand, the transience of the 
canonical trajectory under $\mathbb{Q}^{x_0, \varphi}_x$, stated in Proposition 
\ref{NRY4}, implies that 
$$ \mathbb{Q}^{x_0, \varphi}_x \left(\sum_{k=n}^{\infty} \mathds{1}_{X_k = x_0} = \infty \right) = 0,$$
which completes the proof that 
$\mathbb{Q}_x = \mathbb{Q}^{x_0, \varphi}_x$. 
The transience of the canonical process under $\mathbb{Q}_x = \mathbb{Q}^{x_0, \varphi}_x$ for all $x \in E$ means
that \ref{3} is satisfied for all $x_0, x_1 \in E$. It only remains to check that 
\eqref{2} holds for all $x, x_0 \in E$, i.e. that for all $x, y \in E$, 
$$\mathbb{Q}^{x_0, \varphi}_x (\forall k \geq 0, X_k \neq y)  < \infty.$$
If $g_{x_0}$ denotes the last hitting time of $x_0$ by the canonical process, which is
finite almost everywhere since the process is transient, we get: 
\begin{align*}
\mathbb{Q}^{x_0, \varphi}_x (\forall k \geq 0, X_k \neq y) & 
= \mathbb{Q}^{x_0, \varphi}_x (\forall k \geq 0, X_k \notin \{x_0, y \})
+ \sum_{n \geq 0} \mathbb{Q}^{x_0, \varphi}_x ( g_{x_0} = n,  \forall k \geq 0, X_k \neq y)
\\ &  \leq
\mathbb{Q}^{x_0, \varphi}_x (\forall k \geq 0, X_k \neq x_0)
\\ & + \sum_{n \geq 0} 
\mathbb{Q}^{x_0, \varphi}_x (X_n = x_0,  \forall k \in \{0,1, \dots, n\}, X_k \neq y, \forall k \geq n+1, X_k \neq x_0)
\\ & = \varphi (x) + \sum_{n \geq 0} 
\mathbb{E}_x [\mathds{1}_{X_n = x_0, 
\forall k \in \{0, 1, \dots, n\}, X_n \neq y} \varphi(X_{n+1})]
\\ & = \varphi(x) + 
\sum_{n \geq 0} \mathbb{E}_x [\mathds{1}_{X_n = x_0, 
\forall k \in \{0, 1, \dots, n\}, X_n \neq y} \mathbb{E}_x [\varphi(X_{n+1} )| \mathcal{F}_n]].
\end{align*}
 Now, on the event $X_n = x_0$, the conditional expectation of $\varphi(X_{n+1})$ given $\mathcal{F}_n$ is equal to 
 $$K := \mathbb{E}_{x_0} [\varphi(X_1)] 
 = \sum_{y \in E} p_{x_0, y} \varphi(y),$$
 where $K$ is finite since $p_{x_0, y} = 0$ for all but finitely many $y \in E$. 
 If $T_{y}$ denotes the first hitting time of $y$ by the canonical trajectory, we then get: 
$$\mathbb{Q}^{x_0, \varphi}_x (\forall k \geq 0, X_k \neq y)  \leq \varphi(x) + K 
\mathbb{E}_x \left[ \sum_{n \geq 0} 
\mathds{1}_{X_n = x_0, T_y > n} \right] 
  = \varphi(x) + K \mathbb{E}_{x} 
  [L_{T_y-1}^{x_0}],$$
  where $L_n^x$ denotes the number of hitting times of $x$ at or before time $n$.  
  It is then sufficient to check that 
  $\mathbb{E}_{x} [L_{T_y - 1}^{x_0}]$ is finite. Now, if for $p \geq 1$, $\tau_{p}^{x_0}$ denotes the $p$-th hitting time of $x_0$, we get, using the strong Markov property: 
  $$\mathbb{P}_x [L_{T_y - 1}^{x_0} \geq p]  = \mathbb{P}_x [\tau_{p}^{x_0} < T_y]  = \mathbb{P}_x [\tau_{1}^{x_0} < T_y] \left(\mathbb{P}_{x_0} [\tau_{2}^{x_0} < T_y] \right)^{p-1}
   \leq P^{p-1}$$
  where 
  $$P = \mathbb{P}_{x_0} [\tau_2^{x_0} 
  < T_y].$$
  It is not possible that $P = 1$, otherwise, by the strong Markov property, 
  $$ \mathbb{P}_{x_0} [n-1 < T_y] \geq \mathbb{P}_{x_0} [\tau_n^{x_0} 
  < T_y] = 1$$
  for all $n \geq 1$, and then the canonical trajectory would never hit $y$ under 
  $\mathbb{P}_{x_0}$, which contradicts the  fact that the Markov chain is irreducible and recurrent. Now, since $P <1$, the tail of the law of $L_{T_y-1}^{x_0}$ 
  under $\mathbb{P}_x$ is exponentially decreasing, which implies that 
   $$\mathbb{E}_{x} [L_{T_y - 1}^{x_0}]
   < \infty$$
   and then 
   $$ \mathbb{Q}^{x_0, \varphi}_x (\forall k \geq 0, X_k \neq y) < \infty.$$
\end{proof}
A corollary of Proposition \ref{2.4} is the following result, already contained in  Theorem 4.2.5 of \cite{NRY}:
\begin{corollary}
Let $x_0, x_1 \in E$, and let $\varphi_{x_0}$ be 
a function from $E$ to $\mathbb{R}_+$ such that $\varphi_{x_0}(x_0)= 0$ and 
$\mathbb{E}_{x} [\varphi_{x_0}(X_1)] = \varphi_{x_0}(x)$ for all $x \neq x_0$. 
Then, the function $\varphi_{x_1}$ given by 
$$\varphi_{x_1} (x) := 
\mathbb{Q}_x^{x_0, \varphi_{x_0}} \left(\forall 
k \geq 0, X_k \neq x_1 \right)$$
vanishes at $x_1$, takes finite values and is harmonic at any other point than $x_1$. Moreover, we have, for all $x \in E$, the equality of measures 
$$\mathbb{Q}_{x}^{x_1, \varphi_{x_1}}
= \mathbb{Q}_{x}^{x_0, \varphi_{x_0}}.$$
\end{corollary}
\begin{proof}
We know that \eqref{1}, \eqref{2}, \eqref{3} are satisfied for
$\mathbb{Q}_x = \mathbb{Q}_{x}^{x_0, \varphi_{x_0}}$ and $\varphi = \varphi_{x_0}$. By the 
last item of Proposition \ref{2.4}, 
\eqref{2} and \eqref{3} are still satisfied if we replace $x_0$ by $x_1$, i.e. 
$$\varphi_{x_1} (x) := 
\mathbb{Q}_x \left(\forall 
k \geq 0, X_k \neq x_1 \right) < \infty$$
and 
$$\mathbb{Q}_{x_1} \left( \sum_{k=0}^{\infty}
 \mathds{1}_{X_k = x_1} = \infty \right) = 0.$$
 Now, from the second item of Proposition 
 \ref{2.4}, $\varphi_{x_1}$ vanishes at 
 $x_1$ and is harmonic at any other point, 
 and $\mathbb{Q}_x = \mathbb{Q}_x^{x_1, \varphi_{x_1}}$.  
\end{proof}
From this corollary, we see that in order  to describe
a family of measures of the form $(\mathbb{Q}^{x_0, \varphi_{x_0}}_{x})_{x \in E}$, the role of $x_0$ can be taken by any point in $E$, so the choice of $x_0$ is 
not so important. In the next section, we will clarify this phenomenon, by studying the link between the measures of the form
$\mathbb{Q}^{x_0, \varphi}_{x}$, and the 
Martin boundary of the Markov chain induced 
by $(\mathbb{P}_x)_{x \in E}$. 
We will use the following definition: 
\begin{definition}
We will say that a family $(\mathbb{Q}_x)_{x \in E}$ of $\sigma$-finite measures on $(E^{\mathbb{N}_0}, \mathcal{F}_{\infty})$ is in the {\it class $\mathcal{Q}$}, with respect to 
$(\mathbb{P}_x)_{x \in E}$, if and only if
\eqref{1}, \eqref{2} and \eqref{3} hold for all $x, x_0, x_1, y \in E$ and
$n \geq 0$, or equivalently, iff
it is of the form
$(\mathbb{Q}^{x_0, \varphi}_x)_{x \in E}$
for some $x_0 \in E$, and for some
 function $\varphi$ which is nonnegative, equal to zero at $x_0$ and harmonic for $(\mathbb{P}_x)_{x \in E}$ at any point different from $x_0$. 
\end{definition}
 \section{Link with the Martin boundary}
  \label{martin} 
  In \cite{M}, Martin proves that one can describe all the nonnegative harmonic functions 
  on a sufficiently regular domain of $\mathbb{R}^d$, by a formula which generalizes the Poisson integral formula, available for the harmonic functions on the unit disc. 
This construction has been adapted   
 to the setting of transient Markov chains
 by Doob \cite{D} and Hunt \cite{H}, and 
 then to the setting of recurrent Markov chains by Kemeny and Snell in \cite{KS}, 
 and by Orey in \cite{O}. The construction is also described in a survey by Woess 
 (see \cite{W}, Section 7.H.). 
 
  Let us first recall a possible construction of the Martin boundary, for a transient Markov chain on the countable set $E$. 
  For $x, y \in E$, let $q_{x,y}$ be the transition probability of the Markov chain from $x$ to $y$, and let $G$ be the Green function:  
  $$G(x,y) = \sum_{k=0}^{\infty} 
  (q^k)_{x,y}$$
  where $q^k$ is defined inductively by 
  $$(q^0)_{x,y} = \mathds{1}_{x=y}, 
  \; (q^{k+1})_{x,y} = 
  \sum_{z \in E} (q^k)_{x,z} q_{z,y}.$$
 Let us fix $x_0 \in E$, and let us assume 
 that $G(x_0, y)> 0$ for all  $y \in E$,
 i.e. any state in $E$ is accesible from 
 $x_0$ by the Markov chain. 
 Let $K_{x_0}$ be the function, from $E$ to $\mathbb{R}_+$, given by 
  $$K_{x_0}(x,y) = \frac{G(x,y)}{G(x_0,y)}.$$ 
  One can prove that 
  $$C_{x_0}(x) := \sup_{y \in E} K_{x_0}
  (x,y) < \infty$$
and then, if $w = (w_x)_{x \in E}$ is a summable family of elements in $\mathbb{R}_+^*$, one can define a distance $\rho_{x_0, w}$ on $E$ by 
$$\rho_{x_0, w}(x,y) := \sum_{z \in E}
w_z \, \frac{|K_{x_0} (z,x) - K_{x_0} (z,y)| + |\mathds{1}_{z = x} - \mathds{1}_{z = y}|}{C_{x_0} (z) + 1}.$$
The {\it Martin compactification} of $E$ is 
the topological space $\widehat{E}$, induced by the completion of the metric space $(E, \rho_{x_0, w})$: up to  homeomorphism, $\widehat{E}$ does not depend on the choice of $w$ and the point $x_0$
such that $G(x_0,y) > 0$ for all $y \in E$. The space $\widehat{E}$ is 
compact, its subspace $\partial E := \widehat{E} \backslash E$ is 
a closed set in $\widehat{E}$, called the 
{\it Martin boundary} of $E$. 

If $G(x_0, y) > 0$ for all $y \in E$, and 
if $x \in E$, then the function 
$y \mapsto K_{x_0} (x, y)$
is Lipschitz (with a constant at most
$[1 + C_{x_0} (x)]/w_x$), and then the function $K_{x_0}$ from $E \times E$ to 
$\mathbb{R}_+$ can be uniquely extended by continuity to the set $E \times \widehat{E}$. 
For all $\alpha \in \widehat{E}$, the 
function $x \mapsto K_{x_0}(x,\alpha)$ is 
superharmonic for the  
transition probabilities $(q_{x,y})_{x,y
\in E}$  i.e. for all $x \in E$, 
$$K_{x_0}(x, \alpha) \geq \sum_{y \in E} 
q_{x,y} K_{x_0}(y, \alpha),$$
and it can be harmonic only for $\alpha \in \partial E$. We define the {\it 
minimal boundary} of $E$ as the set
$\partial_m E$ of points 
$\alpha \in \partial E$, such that 
the function  $x \mapsto K_{x_0}(x,\alpha)$
is {\it minimal harmonic}, i.e. it is harmonic, 
and for any  harmonic function $\psi : E \rightarrow \mathbb{R}$ such that 
$0 \leq \psi(x) \leq K_{x_0}(x,\alpha)$ for all 
$x \in E$, there exists $c \in [0,1]$ such that $\psi(x) = c K_{x_0}(x, \alpha)$ for all 
$x \in E$. 
 The following result holds: 
\begin{proposition} \label{31}
The set $\partial_m E$ is a Borel subset of $\partial E$ which, up to canonical homeomorphism, does not depend 
on the choice of $x_0$. Moreover, for any 
choice of $x_0$, a 
nonnegative function $\psi$ from 
$E$ to $\mathbb{R}$ is harmonic if and only if there exists a finite measure
$\mu_{\psi, x_0}$  on 
$\partial_m E$, such that for all $x \in E$, 
$$\psi(x) = \int_{\partial_m E} 
K_{x_0}(x, \alpha) d \mu_{\psi, x_0}(\alpha).$$
If it exists, the mesure $\mu_{\psi,x_0}$ is uniquely determined. 
\end{proposition}

Let us now go back to the assumptions of Section \ref{setting}. In this setting,  
the canonical process $(X_n)_{n \geq 0}$ 
is irreducible and recurrent under $\mathbb{P}_x$ for all 
$x \in E$, and all the  nonnegative harmonic functions are constant. 
Indeed, if $\psi : E \rightarrow \mathbb{R}_+$ is harmonic, $(\psi(X_n))_{n \geq 1}$ is a nonnegative 
martingale, and then it converges a.s., which 
is only possible for $\psi$ constant, since 
$(X_n)_{n \geq 1}$ hits all the points of $E$ infinitely often. 
Then, the definition of the Martin boundary should be modified in order to give a non-trivial result. 
The idea is to kill the Markov chain at some time in order to get a finite Green function. The time which is chosen 
occurs just before the first strictly positive hitting time of some $x_0 \in E$. 
The Green function we obtain in this way is  given by 
\begin{equation} G_{x_0}(x,y) := 
\mathbb{E}_x \left[ L^{y}_{T'_{x_0}-1}
\right], \label{green}
\end{equation}
where 
$$T'_{x_0} := \inf \{n \geq 1, X_n  = x_0\}.$$
Recall that $L^x_n$ denotes the number of hitting times of $x$ at 
and before time $n$. It is easy to check, using the strong Markov property, that the tail of the distribution of $L_{T'_{x_0}}^y$ is exponentially decreasing, which implies that $G_{x_0}(x,y)$ is finite. 
Moreover, 
$G_{x_0}(x_0,y)$ is strictly positive, since all the states in $E$ are accessible from $x_0$ (recall that the Markov chain is irreducible), and then they are also accessible without returning to $x_0$. 
Hence, one can define, similarly as $K_{x_0}(x,y)$ in the transient case: 
$$L_{x_0} (x,y) := 
\frac{G_{x_0}(x,y)}{G_{x_0} (x_0,y)}.$$
The function $L_{x_0}$ induces a distance $\delta_{x_0,w}$ on $E$, given by 
$$\delta_{x_0,w} (x,y)
:= 
\sum_{z \in E}
w_z \, \frac{|L_{x_0} (z,x) - L_{x_0} (z,y)| + |\mathds{1}_{z = x} - \mathds{1}_{z = y}|}{D_{x_0} (z) + 1}.$$
where, as before, $w := (w_x)_{x \in E}$, 
and where
$$D_{x_0}(z) :=   \sup_{y \in E} 
L_{x_0} (z,y) < \infty.$$
The completion of $(E, \delta_{x_0,w})$ induces a topological space $\overline{E}$, called, as before, the 
Martin compactification of $E$: it is possible to prove that the topological structure
of $\overline{E}$ does not depend on $w$ and $x_0$.
 
The transitions of the Markov chain killed just before going to $x_0$ at of after time $1$ are given by 
$(p_{x,y} \mathds{1}_{y \neq x_0})_{x,y \in E}$. A function $\widetilde{\varphi}$ from $E$ to 
$\mathbb{R}_+$ is harmonic with respect to these transitions if an only if the function $\varphi$ given by 
$$\varphi(x) = \widetilde{\varphi} (x) \mathds{1}_{x \neq x_0}$$
is harmonic for the initial Markov chain at any point except $x_0$, and if 
$$ \widetilde{\varphi} (x_0) = \sum_{y \in E} p_{x,y} \varphi(y) = \mathbb{E}_x [\varphi(X_1)].$$
The map going from $\widetilde{\varphi}$ to $\varphi$ is linear and bijective.
By continuity, one can extend $L_{x_0} (x, \alpha)$ to all $x \in E$ and $\alpha \in \overline{E}$. For 
$\alpha$ fixed this function 
is, as in the transient case, superharmonic with respect to the transitions
$(p_{x,y} \mathds{1}_{y \neq x_0})_{x,y \in E}$, and it can only be harmonic for $\alpha$ in the boundary $\partial E$ 
of 
$\overline{E}$, which is, as in the transient case, called the Martin boundary of the Markov chain. 
The minimal boundary $\partial_m E$ is the set of $\alpha \in \partial E$ such that 
$x \mapsto L_{x_0} (x,\alpha)$ is minimal harmonic for $(p_{x,y} \mathds{1}_{y \neq x_0})_{x,y \in E}$. As in 
the transient case, one can show that all harmonic functions for $(p_{x,y} \mathds{1}_{y \neq x_0})_{x,y \in E}$
can be written, in a unique way, as the integral of $x \mapsto L_{x_0} (x, \alpha)$ with respect to
$d\mu (\alpha)$, $\mu$ being a measure on the minimal boundary $\partial_m E$. Stating this precisely, and 
writing this in terms of $\varphi$ rather than $\widetilde{\varphi}$ gives the following: 
 \begin{proposition} \label{decompositionphi} 
The set $\partial_m E$ is a Borel subset of $\partial E$ which does not depend 
on the choice of $x_0$. Moreover, for any $x_0 \in E$, a 
nonnegative function $\varphi$ from 
$E$ to $\mathbb{R}$, such that $\varphi(x_0) = 0$, satisfies $\mathbb{E}_x [\varphi(X_1)] = \varphi(x)$
for all $x \neq x_0$
if and only if there exists a finite measure
$\mu_{\varphi, x_0}$  on 
$\partial_m E$, such that for all $x \neq x_0$, 
$$\varphi(x) = \int_{\partial_m E} 
L_{x_0}(x, \alpha) d \mu_{\varphi, x_0}(\alpha).$$
If it exists, the mesure $\mu_{\varphi,x_0}$ is uniquely determined, and has total mass equal to 
$$\widetilde{\varphi}(x_0) := \mathbb{E}_{x_0} [\varphi(X_1)].$$
\end{proposition}

Now, we can use this result in order to classify the families of 
measures $(\mathbb{Q}^{x_0, \varphi}_x)_{x \in E}$ introduced in Section \ref{setting}. 

Since the Markov chain is irreducible and recurrent, it admits a nonnegative stationnary measure, which is 
unique up to a multiplicative constant. If we fix the constant of normalization, one gets
a function $\beta$ from 
$E$ to $\mathbb{R}_+$, such that for all $y \in E$, 
$$\beta(y) = \sum_{y \in E} p_{x,y} \beta(x).$$
Moreover, the function $\beta$ never vanishes. 
One then gets the following result: 
\begin{proposition}
 For $\alpha \in \partial_m E$, and for all $x_0 \in E$, the function 
 $$\varphi_{x_0, \alpha} : x \mapsto  \frac{L_{x_0}(x,\alpha)}{\beta(x_0)} \mathds{1}_{x \neq x_0},$$
 which vanishes at $x_0$, is harmonic at every point except $x_0$. Moreover, the 
 family of $\sigma$-finite measures
 $(\mathbb{Q}^{x_0, \varphi_{x_0,\alpha}}_x)_{x \in E}$ does not depend on $x_0$. 
\end{proposition}
\begin{proof}
The fact that $\varphi_{x_0, \alpha}$ is harmonic everywhere except at $x_0$ comes directly from the definition of 
the minimal Martin boundary. Moreover, if $x_0, x_1 \in E$, we have proven in Proposition 4.2.10 of \cite{NRY} that 
$\mathbb{Q}^{x_0, \varphi_{x_0,\alpha}}_x = \mathbb{Q}^{x_1, \varphi_{x_1,\alpha}}_x$ for all $x \in E$, if and only if 
for all $\epsilon \in (0,1)$, there exists $A > 0$ such that for all $x \in E$, 
$ \varphi_{x_0,\alpha}(x) + \varphi_{x_1,\alpha}(x) \geq A$ implies 
$$(1-\epsilon)  \varphi_{x_0,\alpha}(x) <  \varphi_{x_1,\alpha}(x) <(1+\epsilon)  \varphi_{x_0,\alpha}(x).$$
One easily checks that this condition is implied by:  
$$\sup_{x \in E}  |\varphi_{x_0,\alpha}(x)  - \varphi_{x_1,\alpha}(x)| < \infty.$$
It is then sufficient to prove this bound for all $x_0, x_1 \in E$ such that $x_0 \neq x_1$. 
Now, a classical construction of the stationnary measure $\beta$ implies, using \eqref{green}, that 
$$G_{x_0} (x_0, y) = \frac{\beta(y)}{\beta(x_0)},$$
which implies, for all $x \neq x_0$, 
$$\varphi_{x_0,\alpha}(x) = \frac{ L_{x_0} (x, \alpha)}{\beta(x_0)} = 
\lim_{y \rightarrow \alpha, y \in E } \frac{G_{x_0} (x, y)}{\beta(x_0)G_{x_0} (x_0,y)} 
= \lim_{y \rightarrow \alpha, y \in E }\frac{G_{x_0} (x, y)}{\beta(y)}, $$
and for 
$x \neq x_1$,
$$ \varphi_{x_1,\alpha}(x)  =  \lim_{y \rightarrow \alpha, y \in E }\frac{G_{x_1} (x, y)}{\beta(y)}.$$
It is then sufficient to prove
$$ \sup_{x \in E \backslash \{x_0, x_1\}, y \in E} \frac{|G_{x_0} (x, y) - G_{x_1} (x, y) |}{\beta(y)} < \infty.$$
Let $G_{x_0,x_1}$ be the Green function of the Markov chain corresponding to $(\mathbb{P}_x)_{x \in \mathbb{R}}$, 
killed just before its first strictly positive hitting time of the set $\{x_0,x_1\}$: 
$$G_{x_0,x_1} (x,y) = \mathbb{E}_x [  L^{y}_{(T'_{x_0} \wedge T'_{x_1})-1}
],$$
where $T'_z$ is the first strictly positive hitting time of $z$.
It is sufficient to prove 
$$\sup_{x \in E \backslash \{x_0, x_1\}, y \in E} \frac{|G_{x_0} (x, y) - G_{x_0,x_1} (x, y) |}{\beta(y)} < \infty$$
and 
$$\sup_{x \in E \backslash \{x_0, x_1\}, y \in E} \frac{|G_{x_1} (x, y) - G_{x_0,x_1} (x, y) |}{\beta(y)} < \infty.$$
Let us show the first bound: the second is obtained by exchanging $x_0$ and $x_1$. 
If $\tau_p^{x_1}$ denotes the $p$-th hitting time of $x_1$, one gets for all 
$x \in E \backslash \{x_0, x_1\}, y \in E$, 
\begin{align*} G_{x_0} (x, y) & = \mathbb{E}_x \left[ \sum_{n=0}^{\infty} 
\mathds{1}_{X_n \in y, n < T'_{x_0} } \right] 
\\ & = \mathbb{E}_x \left[ \sum_{n = 0}^{\tau_1^{x_1} - 1} \mathds{1}_{X_n \in y, n < T'_{x_0}} \right]
+ \mathbb{E}_x \left[ \sum_{p=1}^{\infty} 
 \sum_{n = \tau_p^{x_1}}^{  \tau_{p+1}^{x_1} - 1} 
\mathds{1}_{X_n \in y, n < T'_{x_0}} \right] 
\end{align*}
Since $x \neq x_1$ and then $\tau_1^{x_1 } = T'_{x_1}$, the first term of the last sum is exactly $G_{x_0,x_1} (x, y)$, and by the strong Markov property: 
$$ \mathbb{E}_x \left[\sum_{n = \tau_p^{x_1}}^{  \tau_{p+1}^{x_1} - 1} 
\mathds{1}_{X_n \in y, n < T'_{x_0}} \right] 
= \mathbb{P}_x [ \tau_p^{x_1} < T'_{x_0}] \mathbb{E}_{x_1} 
\left[ \sum_{n = 0}^{T'_{x_1} - 1} \mathds{1}_{X_n \in y, n < T'_{x_0}
} \right] 
= G_{x_0,x_1} (x_1,y) \mathbb{P}_x [ \tau_p^{x_1} < T'_{x_0}] .$$
Hence,
 \begin{align*} G_{x_0} (x, y) & =  G_{x_0,x_1} (x, y) + G_{x_0,x_1} (x_1,y)
  \sum_{p = 1}^{\infty}  \mathbb{P}_x [ \tau_p^{x_1} < T'_{x_0}]
  \\ & =  G_{x_0,x_1} (x, y) + G_{x_0,x_1} (x_1,y) \mathbb{E}_x [ L_{T'_{x_0} - 1}^{x_1}] 
  \\ & = G_{x_0,x_1} (x, y) + G_{x_0,x_1} (x_1,y) G_{x_0} (x,x_1).
 \end{align*}
It is then sufficient to check 
$$ \sup_{x \in E \backslash \{x_0, x_1\}, y \in E} \frac{G_{x_0,x_1} (x_1,y) G_{x_0} (x,x_1)}{\beta(y)} <\infty.$$
Now, 
$$G_{x_0,x_1} (x_1, y) \leq G_{x_1} (x_1, y) = \frac{\beta(y)}{\beta(x_1)}$$
and using the Markov property at the first hitting time of $x_1$, 
$$ G_{x_0} (x,x_1) \leq G_{x_0} (x_1,x_1),$$
which implies 
$$  \sup_{x \in E \backslash \{x_0, x_1\}, y \in E} \frac{G_{x_0,x_1} (x_1,y) G_{x_0} (x,x_1)}{\beta(y)}
\leq \frac{ G_{x_0} (x_1,x_1)}{\beta(x_1)} < \infty.$$
\end{proof}
Since the normalization of $\beta$ is supposed to be fixed, the result we have just proven allows to write, for
all $\alpha \in \partial_m E$, 
$$ \mathbb{Q}^{\alpha}_x := \mathbb{Q}^{x_0, \varphi_{x_0,\alpha}}_x,$$
since the right-hand side does not depend on $x_0 \in E$. 

Using the minimal boundary, one deduces a complete classification of the families of $\sigma$-finite measures 
in the class $\mathcal{Q}$. 
\begin{proposition}
Let $(\mathbb{Q}_x)_{x \in E}$ be a family of $\sigma$-finite measures on 
$(E^{\mathbb{N}_0},\mathcal{F}_{\infty})$. Then $(\mathbb{Q}_x)_{x \in E}$ is 
in the class $\mathcal{Q}$ if and only if there exists a
finite measure $\mu$ on $\partial_m E$ such that for all $A \in \mathcal{F}_{\infty}$, $x \in E$, 
$$\mathbb{Q}_x(A) = \int_{\partial_m E} \mathbb{Q}^{\alpha}_x (A) d \mu(\alpha).$$
In this case, $\mu$ is uniquely determined. 
\end{proposition}

\begin{proof}
 Let us assume $\mathbb{Q}_x = \mathbb{Q}^{x_0, \varphi}_x$ for all $x \in E$. 
 By Proposition \ref{decompositionphi}, there exists a finite measure $\mu_{\varphi,x_0}$ on 
 $\partial_m E$ such that for all $x \in E$, 
 $$\varphi(x) = \mathds{1}_{x \neq x_0} 
 \int_{\partial_m E} L_{x_0}(x,\alpha) d \mu_{\varphi,x_0} (\alpha).$$
 For all $n \geq 1$, and for all nonnegative, $\mathcal{F}_n$-measurable functionals $F_n$, one has 
 \begin{align}
  \mathbb{Q}_x [F_n \mathds{1}_{\forall k \geq n, X_k \neq x_0} ]
  & = \mathbb{E}_x [ F_n \varphi(X_n)]  
\nonumber  \\ & = \mathbb{E}_x \left[ F_n \mathds{1}_{X_n \neq x_0}
  \int_{\partial_m E} L_{x_0}(X_n,\alpha) d \mu_{\varphi,x_0} (\alpha) \right] 
 \nonumber  \\ & =  \mathbb{E}_x \left[ F_n \int_{\partial_m E} \beta(x_0) \varphi_{x_0, \alpha}(X_n) d \mu_{\varphi,x_0} (\alpha) \right] 
\nonumber \\ & =  \int_{\partial_m E} \mathbb{E}_x[F_n \varphi_{x_0, \alpha}(X_n)] d (\beta (x_0) \mu_{\varphi,x_0} (\alpha))
\nonumber \\ & = \int_{\partial_m E}  \mathbb{Q}^{\alpha} \left[F_n \mathds{1}_{\forall k \geq n, X_k \neq x_0} \right]
d (\beta (x_0) \mu_{\varphi,x_0} (\alpha)) \label{5}
  \end{align}
Using the monotone class theorem and the fact that the canonical process is transient 
under $\mathbb{Q}_x$ and $\mathbb{Q}^{\alpha}$, one deduces, for all $A \in \mathcal{F}_{\infty}$, 
$$\mathbb{Q}_x(A) = \int_{\partial_m E} \mathbb{Q}^{\alpha}_x (A) d \mu(\alpha),$$
where the measure 
$$\mu := \beta(x_0) \mu_{\varphi,x_0} (\alpha)$$ is
 finite.
Let us prove the uniqueness of $\mu$. If for two finite measures $\mu$ and $\nu$, and 
for all $A \in \mathcal{F}_{\infty}$,
$$\int_{\partial_m E} \mathbb{Q}^{\alpha}_x (A) d \mu(\alpha) 
= \int_{\partial_m E} \mathbb{Q}^{\alpha}_x (A) d \nu(\alpha),$$
then for  $x_0 \in E$, 
$$ \int_{\partial_m E} \mathbb{Q}^{x_0, \varphi_{x_0, \alpha}}_x (A) d \mu(\alpha) 
= \int_{\partial_m E} \mathbb{Q}^{x_0, \varphi_{x_0, \alpha}}_x (A) d \nu(\alpha),$$
and for all $n \geq 1$,  $B_n \in \mathcal{F}_n$, one gets, by taking 
$A = B_n \cap \{\forall k \geq n, X_k \neq x_0\}$,
$$ \int_{\partial_m E} \mathbb{E}_x [ \mathds{1}_{B_n} \varphi_{x_0, \alpha}(X_n)]  d \mu(\alpha) 
= \int_{\partial_m E} \mathbb{E}_x [ \mathds{1}_{B_n} \varphi_{x_0, \alpha}(X_n)]  d \nu(\alpha),$$
i.e. 
$$\mathbb{E}_x [ \mathds{1}_{B_n} \varphi_1 (X_n)] = \mathbb{E}_x [ \mathds{1}_{B_n} \varphi_2 (X_n)],$$
where 
$$ \varphi_1 (x) =  \int_{\partial_m E} \varphi_{x_0, \alpha}(x)  d \mu(\alpha) $$
and 
$$  \varphi_2 (x) =  \int_{\partial_m E} \varphi_{x_0, \alpha}(x)  d \nu(\alpha).$$
For $y \in E$, taking $B_n = \{X_n = y\}$ gives
$$ \mathbb{E}_x[\mathds{1}_{X_n = y} \varphi_1 (X_n)] =  \mathbb{E}_x[\mathds{1}_{X_n = y} \varphi_2 (X_n)],$$
i.e. 
$$\varphi_1 (y) \mathbb{P}_x [ X_n = y] = \varphi_2 (y) \mathbb{P}_x [ X_n = y].$$
Since the Markov chain is irreducible, there exists $n \geq 1$ such that $\mathbb{P}_x [X_n = y ] > 0$, which implies 
that $\varphi_1(y) = \varphi_2(y)$, i.e. for all $x \in E$, 
$$  \int_{\partial_m E} \varphi_{x_0, \alpha}(x)  d \mu(\alpha) = \int_{\partial_m E} \varphi_{x_0, \alpha}(x)
 d \nu(\alpha) ,$$
and then for $x \neq x_0$, 
$$ \int_{\partial_m E} \frac{L_{x_0}(x,\alpha)}{\beta(x_0)} d \mu(\alpha) 
= \int_{\partial_m E} \frac{L_{x_0}(x,\alpha)}{\beta(x_0)} d \nu(\alpha).$$
The uniqueness given in Proposition \ref{decompositionphi} implies that $\mu = \nu$. 

It remains to show that any family $(\mathbb{Q}_x)_{x \in E}$ of measures such that for all 
$A  \in \mathcal{F}_{\infty}$, 
$$\mathbb{Q}_x(A) = \int_{\partial_m E} \mathbb{Q}^{\alpha}_x (A) d \mu(\alpha)$$
has the form $(\mathbb{Q}^{x_0, \varphi}_x)_{x \in E}$ if $\mu$ is a finite measure on $\partial_m E$. 
Indeed, by reversing the computation given in \eqref{5} and by replacing 
$\mu_{\varphi,x_0}$ by $\mu/(\beta(x_0))$, one deduces that for $F_n$ nonnegative and 
$\mathcal{F}_n$-measurable, 
$$ \mathbb{Q}_x (F_n \mathds{1}_{\forall k \geq n, X_k \neq x_0}) = 
\mathbb{Q}^{x_0, \varphi}_x (F_n \mathds{1}_{\forall k \geq n, X_k \neq x_0}),$$
where 
$$\varphi(x) := \mathds{1}_{x \neq x_0} \int_{\partial_m E} \frac{L_{x_0}(x,\alpha)}{\beta(x_0)} d \mu(\alpha).$$
Since the canonical process is transient under $\mathbb{Q}_x$ and $\mathbb{Q}^{x_0, \varphi}_x$, one deduces
that $\mathbb{Q}_x = \mathbb{Q}^{x_0, \varphi}_x$. 
\end{proof}
The result we have just proven gives a disintegration of all families of measures
in the class $\mathcal{Q}$,
in terms of the families $(\mathbb{Q}^{\alpha}_x)_{x \in E}$ for $\alpha \in 
\partial_m E$. 

\section{Convergence of the canonical process under $\mathbb{Q}_x^{\alpha}$} \label{Qalpha}
 In this section, we study the canonical trajectory under $\mathbb{Q}_x^{\alpha}$, for 
  $\alpha$ in the minimal boundary of the Markov chain corresponding to $(\mathbb{P}_x)_{x \in E}$. The main statement we 
  will prove is the following result of convergence:   
   \begin{proposition} \label{convergencetrajectories}
  For all $x \in E$, and for all $\alpha \in \partial_m E$, $(\mathbb{Q}^{\alpha}_x)$-almost every trajectory 
  tends to $\alpha$ at infinity.
  \end{proposition}
  \begin{proof}
  The proof of this statement will be done in several steps. 
  A difficulty in the study of $\mathbb{Q}_x^{\alpha}$ is the fact that this measure 
  is not finite in general. Hopefully, 
  $\mathbb{Q}_x^{\alpha}$ can be proven to be equivalent to probability measures, which can be explicitly
  described. Moreover, one can choose such a probability measure, in such a way that the corresponding random 
  trajectory is a transient Markov chain. 
 \begin{proposition}
  For $r \in (0,1)$, $x, x_0 \in E$, $\alpha \in \partial_m E$, let 
  $$\psi_{x_0,\alpha,r}(x) := \frac{1}{\beta(x_0)} \left[ \frac{r}{1-r} + L_{x_0}(x,\alpha) \mathds{1}_{x \neq x_0}
  \right].$$ 
  Then, if $L_{\infty}^{x_0}$ denotes the total number of hitting time of $x_0$ by the canonical trajectory, 
  the measure 
  $$\mathbb{P}^{x_0,\alpha,r}_x := \frac{r^{L_{\infty}^{x_0}}}{\psi_{x_0,\alpha,r}(x)} \cdot \mathbb{Q}_x^{\alpha}$$
  is the probability distribution of a Markov chain, starting at $x$, with transition probabilities
  $(q^{x_0, \alpha,r}_{x,y})_{x,y \in
  E}$, where 
  $$q^{x_0, \alpha,r}_{x,y} = \frac{ \psi_{x_0,\alpha,r} (y)}{ \psi_{x_0,\alpha,r} (x)} p_{x,y}$$
  if $x  \neq x_0$, and 
  $$ q^{x_0, \alpha,r}_{x_0,y} = r \frac{ \psi_{x_0,\alpha,r} (y)}{ \psi_{x_0,\alpha,r} (x_0)} p_{x_0,y}.$$
 \end{proposition}
\begin{proof}
 The discussion at the beginning of Chapter 4 of \cite{NRY} shows the following: if $\varphi(x_0) = 0$ and 
 $\varphi$ is harmonic at all points different from $x_0$, then 
 for $$\psi_r (x) := \frac{r}{1-r} \mathbb{E}_{x_0} [\varphi(X_1)] + \varphi(x),$$
 the measure 
 $$\mu^{(r)}_x := r^{L_{\infty}^{x_0}} \cdot \mathbb{Q}_x^{x_0, \varphi}$$
 is finite and satisfies, for all $n \geq 0$, and for all $F_n$ nonnegative, $\mathcal{F}_n$-measurable, 
 $$\mu^{(r)}_x (F_n) := \mathbb{E}_x [ \psi_r (X_n) r^{L_{n-1}^{x_0}} F_n ],$$
 where $L_{n-1}^{x_0}$ is the number of hitting times of $x_0$ at or before time $n-1$. 
 
 In the case we consider here, we have 
 $$\varphi(x) = \varphi_{x_0, \alpha} (x) = \frac{L_{x_0}(x,\alpha)}{\beta(x_0)} \mathds{1}_{x \neq x_0},$$
 and by applying Proposition \ref{decompositionphi} to $\mu_{\varphi, x_0}$ equal to $1/\beta(x_0)$ times the 
 Dirac mass at $\alpha$, 
 $$\mathbb{E}_{x_0} [\varphi(X_1)] = \frac{1}{\beta(x_0)},$$
 the total mass of $\mu_{\varphi, x_0}$.
 Hence, 
$$\psi_r (x) =  \frac{r}{(1-r) \beta(x_0)} +   \frac{L_{x_0}(x,\alpha)}{\beta(x_0)} \mathds{1}_{x \neq x_0}
 = \psi_{x_0,\alpha,r}(x),$$ 
 $$ \mu^{(r)}_x = \psi_{x_0,\alpha,r}(x) \cdot \mathbb{P}^{x_0,\alpha,r}_x,$$
 and then for $n \geq 0$, $F_n$ nonnegative and $\mathcal{F}_n$-measurable, 
 $$ \mathbb{P}^{x_0,\alpha,r}_x (F_n) = \mathbb{E}_x \left[ \frac{\psi_{x_0,  \alpha, r}(X_n)}{\psi_{x_0,
 \alpha, r}(x)}
  r^{L_{n-1}^{x_0}} F_n \right].$$
  Taking $n = 0$ and $F_n = 1$, we deduce that $\mathbb{P}^{x_0,\alpha,r}_x$ is a probability measure. 
 Moreover, for all $y_0, y_1, \dots, y_n \in E$, 
 \begin{align*}
\mathbb{P}^{x_0,\alpha,r}_x (X_0 = y_0, \dots, X_n = y_n) 
& = \mathds{1}_{y_0 = x} \, \left(\prod_{j=0}^{n-1} p_{y_j,y_{j+1}}  \right) 
\, \frac{\psi_{x_0,  \alpha, r} (y_n)}{\psi_{x_0,  \alpha, r} (x)} \, r^{\sum_{j=0}^{n-1} \mathds{1}_{y_j = x_0}}
\\ & = \mathds{1}_{y_0 = x}\left(\prod_{j=0}^{n-1} p_{y_j,y_{j+1}}  \right) \, 
\left(\prod_{j=0}^{n-1}\frac{\psi_{x_0,  \alpha, r} (y_{j+1})}{\psi_{x_0,  \alpha, r} (y_j)}    \right) 
\left(\prod_{j=0}^{n-1} r^{\mathds{1}_{y_j = x_0}} \right)
\\ & = \mathds{1}_{y_0 = x} \, \prod_{j=0}^{n-1} \left(  p_{y_j,y_{j+1}} 
 \frac{\psi_{x_0,  \alpha, r} (y_{j+1})}{\psi_{x_0,  \alpha, r} (y_j)} r^{\mathds{1}_{y_j = x_0}} \right)
  \\ & = \mathds{1}_{y_0 = x} \, \prod_{j=0}^{n-1}   q^{x_0, \alpha,r}_{y_j,y_{j+1}},
  \end{align*}
which proves the desired result. 
\end{proof}
Since the canonical trajectory is transient under $\mathbb{Q}_x^{\alpha}$, 
it is also transient under 
$\mathbb{P}_x^{x_0,\alpha,r}$, since the two measures are absolutely continuous with respect to each other. Hence, one can consider the Martin boundary of the corresponding transient Markov chain. 
If we take $x_0$ as the reference point, we need to consider the Green function 
$G_{x_0, \alpha, r}$ given by
$$G_{x_0, \alpha, r} (x,y)
:= \sum_{k = 0}^{\infty}
 \mathbb{P}_x^{x_0,\alpha,r} (X_k = y),$$
 and the function 
 $K_{x_0, \alpha, r}$ given by 
 $$K_{x_0, \alpha, r}(x,y) := \frac{G_{x_0, \alpha, r} (x,y)}{G_{x_0, \alpha, r} (x_0,y)}.$$
 One then gets the following result: 
  \begin{proposition} \label{equalitymartin} 
  For all $x, y \in E$, one has:
  $$K_{x_0, \alpha, r}(x,y)
  = \frac{\psi_{x_0, \alpha, r} 
  (x_0)}{\psi_{x_0, \alpha, r} 
  (x)} \left(1 + \frac{1-r}{r}
  L_{x_0} (x,y) \mathds{1}_{x \neq x_0}
 \right).$$
 The Martin boundary associated to the 
  transient Markov chain corresponding to $(\mathbb{P}_x^{x_0,\alpha,r})_{x \in E}$ is 
 canonically homeomorphic to the Martin 
 boundary associated to the recurrent Markov chain corresponding to $(\mathbb{P}_x)_{x \in E}$, and the analogous statement is true if we replace the Martin boundary by the minimal boundary. Moreover, the function $1$ is a minimal harmonic function,
 for the Markov chain given by $(\mathbb{P}_x^{x_0, \alpha, r})_{x \in E}$,
 which corresponds to the point $\alpha$ of the minimal boundary $\partial_m E$.  
  \end{proposition}
  \begin{proof}
  For $x, y \in E$, $n \geq 1$,
  \begin{align*}\mathbb{P}_x^{x_0,\alpha,r}
  (X_n = y) & = \mathbb{E}_x 
  \left[ \frac{\psi_{x_0, \alpha, r} 
  (X_n)}{\psi_{x_0, \alpha, r} 
  (x)} r^{L_{n-1}^{x_0}} \mathds{1}_{X_n = y} \right]
   \\ & =  \frac{\psi_{x_0, \alpha, r} 
  (y)}{\psi_{x_0, \alpha, r} 
  (x)}  \mathbb{E}_x 
  \left[  r^{L_{n-1}^{x_0}} \mathds{1}_{X_n = y} \right],
  \end{align*}
  and then 
  $$G_{x_0, \alpha, r} (x,y) 
    = \frac{\psi_{x_0, \alpha, r} 
  (y)}{\psi_{x_0, \alpha, r} 
  (x)}  \mathbb{E}_x \left[ 
  \sum_{n=0}^{\tau^{x_0}_1} \mathds{1}_{X_n = y} + \sum_{p=1}^{\infty} 
  r^p \sum_{n=\tau^{x_0}_p + 1}^{\tau^{x_0}_{p+1}} \mathds{1}_{X_n = y} 
  \right].$$
  Using the strong Markov property, one deduces 
  $$ G_{x_0, \alpha, r} (x,y) 
    = \frac{\psi_{x_0, \alpha, r} 
  (y)}{\psi_{x_0, \alpha, r} 
  (x)} \left(  \mathbb{E}_x \left[L^y_{T_{x_0}} \right] 
  + \frac{r}{1-r}  \mathbb{E}_{x_0} 
  \left[L^y_{T'_{x_0}} - L^y_0 \right]
  \right). $$
  Now, under $\mathbb{P}_{x}$,
  $$L_{T_{x_0}}^y =  
  L_{T_{x_0}-1}^y + \mathds{1}_{y = x_0}
 =  L_{T'_{x_0}-1}^y \mathds{1}_{x \neq x_0}
  + \mathds{1}_{y = x_0} $$
  and under $\mathbb{P}_{x_0}$, 
  $$L_{T'_{x_0}}^y - L_0^y
  = \left(L_{T'_{x_0}-1}^y + \mathds{1}_{y = x_0}\right) - \mathds{1}_{y = x_0}
  = L_{T'_{x_0}-1}^y.$$
  Hence, 
  $$ G_{x_0, \alpha, r} (x,y) 
  =  \frac{\psi_{x_0, \alpha, r} 
  (y)}{\psi_{x_0, \alpha, r} 
  (x)} \left( \mathds{1}_{y = x_0} 
  + G_{x_0} (x,y) \mathds{1}_{x \neq x_0} + 
  \frac{r}{1-r} G_{x_0}(x_0, y) \right).$$
  In particular, 
  $$ G_{x_0, \alpha, r} (x_0,y)
  = \frac{\psi_{x_0, \alpha, r} 
  (y)}{\psi_{x_0, \alpha, r} 
  (x_0)} \left( \mathds{1}_{y = x_0} 
   + 
  \frac{r}{1-r} G_{x_0}(x_0, y) \right).$$
 Taking the quotient of the two expressions, we get, 
  after dividing the numerator and the denominator by $ r G_{x_0} (x_0,y)/(1-r)$,  and by checking separately the cases 
  $y \neq x_0$ and $y = x_0$, 
  $$K_{x_0, \alpha, r}(x,y)
  = \frac{\psi_{x_0, \alpha, r} 
  (x_0)}{\psi_{x_0, \alpha, r} 
  (x)} \left(1 + \frac{1-r}{r}
  L_{x_0} (x,y) \mathds{1}_{x \neq x_0}
 \right).$$
For $x, (y_n)_{n \geq 1}$ in $E$, 
it is then clear that 
$K_{x_0, \alpha, r}(x,y_n)$ converges if 
and only if 
$ L_{x_0} (x,y_n)$ converges: this 
equivalence is also true for $x = x_0$, 
since 
$K_{x_0, \alpha, r}(x_0,y_n)
= L_{x_0} (x_0, y_n) = 1$. 
This equivalence implies the equality, up 
to a canonical homeomorphism, 
 of the Martin boundaries
associated to $(\mathbb{P}_x)_{x \in E}$ 
and $(\mathbb{P}^{x_0, \alpha,r}_x)_{
x \in E}$. 

Let us now check the equality of the minimal boundaries. 
It is straightforward to check that there 
is a bijective map $\mathcal{R}$ from the set of functions $\varphi$ from $E$ to $\mathbb{R}$  
for which $\varphi(x_0)  = 0$ and 
$\mathbb{E}_x [\varphi(X_1)] = \varphi(x)$ if $x \neq x_0$, to the set of functions $h$ from $E$ to $\mathbb{R}$ which are harmonic with respect to the Markov chain
associated to $(\mathbb{P}^{x_0, \alpha,r}_x)_{
x \in E}$. This map is given as follows: 
$$\mathcal{R} (\varphi) (x) 
= \frac{1}{\psi_{x_0, \alpha,r}(x)} 
\left( \varphi(x) + \frac{r}{1-r} 
\mathbb{E}_{x_0} [\varphi(X_1)] \right), $$
and one has 
$$\mathcal{R}^{-1} (h)(x) = 
\psi_{x_0, \alpha,r}(x) h (x) - 
r \mathbb{E}_{x_0} [\psi_{x_0, \alpha,r}(X_1) h (X_1)].$$ 
It is obvious that $\mathcal{R}$ and $\mathcal{R}^{-1}$ are linear maps,
and that $\mathcal{R}$ sends nonnegative functions to nonnegative functions. Moreover this last property is also true 
for $\mathcal{R}^{-1}$. Indeed, if $h$ 
is nonnegative, then $\mathcal{R}^{-1}(h)$ is 
harmonic for $(\mathbb{P}_x)_{x \in E}$ 
at every point except $x_0$, vanishes
 at $x_0$ and is bounded from below by 
$$ - C := - r \mathbb{E}_{x_0} [\psi_{x_0, \alpha,r}(X_1) h (X_1)].$$ 
Hence, for all $x \in E$, 
$$\mathcal{R}^{-1}(h)(x) 
= \mathbb{E}_x [ \mathcal{R}^{-1}(h)(X_{
n \wedge T_{x_0}})]
= \mathbb{E}_x[ \mathcal{R}^{-1}(h)(X_{n})
\mathds{1}_{T_{x_0} > n} ] 
\geq -C \mathbb{P}_x[ T_{x_0} > n],$$
and letting $n \rightarrow \infty$, 
$$\mathcal{R}^{-1}(h)(x) \geq 0,$$
since by the recurrence of the canonical process under $\mathbb{P}_x$, 
$$\mathbb{P}_x [T_{x_0} > n ] 
\underset{n \rightarrow \infty}{\longrightarrow} 0.$$
One deduces that $\mathcal{R}$ and $\mathcal{R}^{-1}$ preserve the minimality of the corresponding harmonic functions. 
Moreover, one has for all $\gamma \in \partial_m E$, 
\begin{align*}\mathcal{R} \left(\varphi_{x_0, \gamma} \right) (x) & = 
\frac{1}{  \psi_{x_0, \alpha, r} (x)} \left( \frac{1}{\beta(x_0)} L_{x_0}(x,\gamma) \mathds{1}_{x 
\neq x_0} + \frac{r}{1-r} \mathbb{E}_{x_0}
[ \varphi_{x_0, \gamma} (X_1)] \right)
\\ & = \frac{1}{  \beta(x_0)\psi_{x_0, \alpha, r} (x)} \left( L_{x_0}(x,\gamma) \mathds{1}_{x 
\neq x_0} + \frac{r}{1-r} \right)
\\ & = \frac{r}{(1-r)\beta(x_0)\psi_{x_0, \alpha, r} (x_0)} K_{x_0, \alpha, r} (x,\gamma).
\end{align*}
In the second equality, we use that 
$$\mathbb{E}_{x_0}
[ \varphi_{x_0, \gamma} (X_1)] = 
\frac{1}{\beta(x_0)},$$
which is a consequence of Proposition 
\ref{decompositionphi} applied to
$1/\beta(x_0)$ times the Dirac measure at 
$\gamma$. 
Since $\gamma \in \partial_m E$, $\varphi_{x_0, \gamma}$ is minimal as a nonnegative function vanishing at $x_0$ and $(\mathbb{P}_x)_{x 
\in E}$-harmonic outside $x_0$, 
$\mathcal{R}(\varphi_{x_0, \gamma})$
is then minimal as a $(\mathbb{P}_x^{x_0, \alpha, r})_{x \in E}$-harmonic function, 
and by the previous computation, 
$x \mapsto K_{x_0, \alpha, r} (x, \gamma)$ si also minimal, which implies that $\gamma$ is also in the minimal boundary 
of $E$ for the transient Markov chain 
$(\mathbb{P}^{x_0, \alpha, r}_x)_{x \in E}$. Using the reverse map $\mathcal{R}^{-1}$, we 
deduce similarly that any point in the minimal boundary for $(\mathcal{P}^{x_0, \alpha,r}_x)_{x \in E}$ is also in the minimal boundary for $(\mathbb{P}_x)_{x \in E}$. 
We have then the identity (up to canonical homeomorphism) between the two minimal boundaries. 

Moreover, for $\gamma = \alpha$, we get the following: 
$$ K_{x_0, \alpha, r} (x, \alpha) =  \frac{\psi_{x_0, \alpha, r} (x_0)}{\psi_{x_0, \alpha, r} (x)} 
\left( 1 + \frac{1 - r}{r} L_{x_0} (x, \alpha) \mathds{1}_{x \neq x_0} \right)
= \left( \frac{1-r}{r} \right) \beta(x_0) \psi_{x_0, \alpha, r} (x_0). 
$$
Hence, if we refer to Proposition \ref{31}, the constant function equal to $1$ can be written as follows: 
$$1 = \int_{\partial_m E}   K_{x_0, \alpha, r} (x, \beta) d \mu(\beta)$$
where 
$$\mu =  \frac{r}{(1-r) \beta(x_0) \psi_{x_0, \alpha, r} (x_0)} \delta_{\alpha},$$
$\delta_{\alpha}$ denoting the Dirac measure at $\alpha$. 
Since $\mu$ is carried by $\alpha$, and $\alpha \in \partial_m E$ by assumption, the last statement
of Proposition \ref{equalitymartin} is proven. 
  \end{proof}
We can now easily finish the proof of Proposition \ref{convergencetrajectories}. 
 Applying Theorem 3.2. of Kemeny and Snell \cite{KS} to the transient Markov chain associated to 
 $\mathbb{P}_x^{x_0, \alpha, r}$ and to the constant harmonic function $h = 1$, and using the last statement of  
 Proposition \ref{equalitymartin}, we deduce that $\mathbb{P}_x^{x_0, \alpha, r}$-almost surely, 
 the canonical trajectory tends to $\alpha$. Since $\mathbb{Q}^{\alpha}_x$ is absolutely continuous 
 with respect to $\mathbb{P}_x^{x_0, \alpha, r}$ (with density $r^{- L_{\infty}^{x_0}}$), the canonical trajectory
 also tends to $\alpha$ under $\mathbb{Q}^{\alpha}_x$.
\end{proof}
From the fact that $(\mathbb{Q}_x^{\alpha})_{x \in E}$ satisfies the condition \eqref{1} and from
 Proposition 
\ref{convergencetrajectories}, we deduce the following informal interpretation: under 
$\mathbb{Q}_x^{\alpha}$, the canonical process corresponds to the Markov chain given by $\mathbb{P}_x$, 
conditioned to tend to $\alpha$ at infinity. Of course, this interpretation is not rigorous since 
$\mathbb{Q}_x^{\alpha}$ is not a probability measure in general, and even not a finite measure. Moreover, 
under $\mathbb{P}_x$, the canonical process is recurrent, so it cannot converge to a point of the Martin boundary.  
 \section{Some examples} \label{examples}
In this section, we look again at the  examples given in Chapter 4 of \cite{NRY}. 
\subsection{The simple random walk on $\mathbb{Z}$} 
The simple random walk on $\mathbb{Z}$
is the Markov chain given by the transition probabilities 
$(p_{x,y})_{x, y \in \mathbb{Z}}$ where 
$p_{x,x+1} = p_{x, x-1} = 1/2$ and 
$p_{x,y} = 0$ if $|y-x| \neq 1$. For all 
$x \in \mathbb{Z}$, $p_{x,y} = 0$ for all 
but finitely many $y \in \mathbb{Z}$, 
and the simple random walk is irreducible and recurrent. We can then do the construction given in Chapter 4 of \cite{NRY} and in the present article. 
If we take $x_0 = 0$, we get, by using 
standard 
martingale arguments,
$$ G_0(0,y) = 1$$
for all $y \in \mathbb{Z}$, 
and 
for all $x \in \mathbb{Z} \backslash \{0\}$ and $y \in \mathbb{Z}$,
$$G_0(x,y) = 2(|x| \wedge |y|)\mathds{1}_{xy > 0}.$$
Hence, $L_0(0,y) = 1$ and for $x  \neq 0$,$$L_0(x,y) = 2(|x| \wedge |y|)\mathds{1}_{xy > 0}.$$
We deduce that the Martin boundary of the
standard random walk has exactly 
two points. We denote these points $-\infty$ and $\infty$, the distinction 
between them being given by the formulas:  
$$L_0(x, \infty) = 2 x_+ + \mathds{1}_{x = 0}, \; 
L_0(x, -\infty) = 2 x_- + \mathds{1}_{x = 0}.$$
This notation is justified by the following fact: a sequence of points in 
$\mathbb{Z}$ tends to $\infty$ in the 
Martin compactification of $\mathbb{Z}$ if 
and only if it tends to $\infty$ in the 
usual sense, and the similar statement is 
true for $-\infty$.  
If we normalize the stationnary measure 
by taking $\beta(x) = 1$ for all $x \in 
\mathbb{Z}$, we get 
$$\varphi_{0,\infty} (x) = 2 x_+, 
\varphi_{0,-\infty} (x) = 2 x_-.$$
The nonnegative functions $\varphi$ such that $\varphi(0) = 0$ and $\varphi$ is 
harmonic at all $x \neq 0$ are exactly 
the linear combinations of 
$\varphi_{0,\infty}$ and $\varphi_{0,-\infty}$ with nonnegative coefficients. 
We deduce that the minimal boundary of 
$\mathbb{Z}$  
is equal to its Martin boundary, i.e. 
has the two points $-\infty$ and $\infty$. 
The families of $\sigma$-finite measures in the class $\mathcal{Q}$ are then exactly the families 
of the form $(\alpha \mathbb{Q}_x^{\infty}
+ \beta \mathbb{Q}_x^{-\infty})_{x \in 
\mathbb{Z}}$ for $\alpha, \beta \geq 0$. 
Hence, we do not obtain other measures than those given in Subsection 4.3.1 of \cite{NRY}. 
\subsection{The simple random walk in 
$\mathbb{Z}^2$}
In this case, we have $E = \mathbb{Z}^2$ 
and the transition probabilities are given 
by $p_{x,y} = 1/4$ if $||x-y|| = 1$ and 
$p_{x,y} = 0$ otherwise. It has been 
shown that in this situation, 
there exists, up to a multiplicative
constant, a unique nonnegative function 
which vanishes at $(0,0)$ and 
which is harmonic everywhere else. 
This property is, for example, stated
in Section 31  of \cite{S} (statement P3), in the case where we replace the simple random walk by 
a general irreducible, recurrent, aperiodic Markov chain in $\mathbb{Z}^2$, 
for which the increments are i.i.d. random 
variables. Since the simple random 
walk is not aperiodic, the result in \cite{S} doesn't 
apply directly. However, it is easy to  
deal with this problem: if $a$ is 
a nonnegative function, vanishing at 
$(0,0)$ and harmonic elsewhere for the 
transitions $(p_{x,y})_{x,y \in \mathbb{Z}^2}$, then for 
$E' := \{(a,b) \in \mathbb{Z}^2, 
 a + b \,  \operatorname{ even }\}$,
the restriction of $a$ to $E'$ is  harmonic, except at $(0,0)$, for the transition $p^2$ obtained
by iterating two steps of the Markov 
chain with transition $p$, i.e.
$(p^2)_{x,y}$ is $1/16$ for $||x-y|| = 2$, 
$1/8$ for $||x-y|| = \sqrt{2}$ and 
$1/4$ for $x = y$. 
This Markov chain on $E'$ is
 irreducible, recurrent, aperiodic, and  
 then the restriction of $a$ to 
 $E'$ is uniquely determined up to a 
 multiplicative constant. Now 
 for $x \in \mathbb{Z}^2 \backslash E'$, 
$a(x)$ is the average of the four numbers
$a(x \pm (0,1))$, $a(x \pm (1,0))$, where  $x \pm (0,1)$ and $x \pm (1,0)$ are in $E'$, so it 
is also uniquely determined. 
We have already written the expression 
of $a$ in Subsection 4.3.5 of \cite{NRY}: 
if the multiplicative constant is 
suitably chosen, then $a$ is the so-called 
{\it potential kernel}, given by 
$$a(x) = \underset{N \rightarrow \infty}{\lim}   \left(\sum_{n = 0}^N \mathbb{P}_{(0,0)}
(X_n = (0,0)) - \sum_{n = 0}^N \mathbb{P}_{x}
(X_n = (0,0))\right).$$
In Section 15 of \cite{S}, some explicit
values of $a$ are given. If $x = (0,0)$, 
then $a(x) = 0$, if
$||x|| = 1$, then $a(x) = 1$, and 
for all $n \geq 1$, 
$$a(\pm n, \pm n) = \frac{4}{\pi} 
\sum_{j= 1}^{n} \frac{1}{2j-1}.$$
Knowing these values is sufficient 
to successively recover all the values of $a$, 
by only using the fact that 
$a$ is harmonic, and that 
$a$ has the same symmetries as the lattice
$\mathbb{Z}^2$. For example, we get
$$4 a(1, 0) = a(2,0) + a(0,0) + 
a(1,1) + a(1,-1),$$
and then
$$a(2,0) = 4 a(1, 0) - a(0,0) - 2 a(1,1) 
= 4 - \frac{8}{\pi}.$$
Similarly, 
$$4 a(1,1) = a(1,2) + a(2,1) + a(1,0) 
+ a(0,1) = 2a(2,1) + 2a(1,0),$$
$$a(2,1) = 2 a(1,1) - a(1,0) = 
 \frac{8}{\pi} - 1$$
 and so on.
In particular, for all $x \in \mathbb{Z}^2$, $a(x) \in \mathbb{Q} + \frac{1}{\pi}\mathbb{Q}$.
The following asymptotics has been given 
by St\"ohr \cite{St}, then improved and 
generalized by Fukai and Uchiyama \cite{F}: 
$$a(x) = \frac{2}{\pi} \log ||x|| + 
\frac{2 \gamma + \log 8}{\pi} + O(1/||x||^2),$$
where $\gamma$ is the Euler-Mascheroni
constant. 

The uniqueness of $a$, up to a multiplicative constant, shows that the simple random walk in $\mathbb{Z}^2$ has a Martin boundary with 
only one point, which can naturally be 
denoted $\infty$.  
If we go back to the definition of the 
Martin boundary given here, we deduce that 
for all $x \in \mathbb{Z}^2 \backslash \{0\}$,
$$\frac{G_{(0,0)} (x,y)}{G_{(0,0)} ((0,0),y)} \underset{||y|| \rightarrow 
\infty}{\longrightarrow} C a(x)$$
for some constant $C > 0$. Since 
the counting measure is invariant for the simple random walk we deduce that ${G_{(0,0)} ((0,0),y)} = 1$ for all 
$y \in \mathbb{Z}^2$, and  
then for $x \neq (0,0)$, 
$$\mathbb{E}_x [L_{T_{(0,0)}}^y] 
\underset{ ||y|| \rightarrow 
\infty}{\longrightarrow} C a(x).$$
Moreover, by the Markov property, for 
$y  \neq (0,0)$,
$$\mathbb{E}_{(0,0)} [L_{T'_{(0,0)}}^y]
= \frac{1}{4} \sum_{x \in \{(0,1), (0,-1), 
(1,0), (-1,0)\}} 
\mathbb{E}_x [L_{T_{(0,0)}}^y],$$
and then, by letting $||y|| \rightarrow
\infty$,
$$1 = \frac{1}{4} \sum_{x \in \{(0,1), (0,-1), 
(1,0), (-1,0)\}} C a(x) = C a(0,1),$$
and then $C = 1$, and 
$$\mathbb{E}_x [L_{T_{(0,0)}}^y] 
\underset{ ||y|| \rightarrow 
\infty}{\longrightarrow} a(x).$$

Since the Martin boundary of $E$ has 
only one point in this example, 
the class $\mathcal{Q}$ contains only
the nonegative multiples of the family 
of measures
$(\mathbb{Q}^{\infty}_x)_{x \in \mathbb{Z}^2}$. 

 The results given here on the simple random walk in $\mathbb{Z}$ or $\mathbb{Z}^2$, its 
 potential kernel and its Martin boundary
 has been adapted to more general random walks on groups. For example, see 
 Kesten \cite{K} or Kesten and Spitzer 
  \cite{KSp}.

\subsection{The "bang-bang random walk"}
This Markov chain is given in Subsection 
4.3.2 of  \cite{NRY}. We have $E = \mathbb{N}_0$, the set of nonnegative integers, and the transition probabilities
are given by $p_{0,1} = 1$, and for all $y \geq 1$,  $p_{y,y+1} = 
q \in (0,1/2)$, $p_{y,y-1} = 1-q$ and $p_{x,y} = 0$ 
for $|y - x | \neq 1$ (in \cite{NRY}, only the
case $q = 1/3$ is considered, but the generalization is straightforward).  
It is easy to check that the Markov chain is irreducible and recurrent. 
Moreover, for $\alpha := (1-q)/q$, 
$(\alpha^{X_{n \wedge T_0}})_{n \geq 0}$ is a martingale under $\mathbb{P}_x$ for
all $n \in \mathbb{N}_0$. 
By a standard martingale argument, one deduces that for $0 < x < y$, 
$$\mathbb{P}_x (T_y < T_0)
= \frac{\alpha^x - 1}{\alpha^y - 1}.
$$
Now, for all $y > 0$, under $\mathbb{P}_y$: 
\begin{itemize}
\item With probability $q$, $ X_1 = y+1$ and then the Markov chain goes 
almost surely back to $y$ before hitting $0$. 
\item With probability $1-q$, $X_1 = y-1$, and then the conditional probability that 
the Markov chain goes to $0$ before returning 
to $y$ is
$$\mathbb{P}_{y-1} (T_y > T_0)
= 1 - \frac{\alpha^{y-1} - 1}{\alpha^y - 1}.$$
\end{itemize}
Hence, the probability that the Markov hits $0$ before returning to $y$ is 
$$ \mathbb{P}_y [T_0 < \tau_2^y] 
= (1-q) \frac{\alpha^{y} - \alpha^{y-1}}
{\alpha^y - 1} = (1-q) (1 - (1/\alpha))
\frac{\alpha^{y}}{\alpha^y - 1} = 
\frac{(1-2q) \alpha^{y}}{\alpha^y - 1}.$$
If we choose $x_0 = 0$, we deduce 
$$G_0 (y,y) = \frac{\alpha^y - 1}{ 
(1 - 2q) \alpha^y} = \frac{1 - [q/(1-q)]^y}{1 - 2q}.$$
Applying the Markov property to the first hitting time of $y$, we deduce that for 
$ x \geq y > 0$, 
$$G_0(x,y) = G_0 (y,y)$$
and for $y > 0$, $x \leq y$, 
$$G_0(x,y) = \frac{\alpha^x - 1}{ 
(1 - 2q) \alpha^y}.$$
If $x > 0$, one obviously has
$$G_0(x,0) = 0,$$
one has 
$$G_0(0,0) = 1,$$
and for all $ y > 0$, 
$$G_0(0,y) = G_0(1,y) = \frac{\alpha - 1}
{(1-2q) \alpha^y} = \frac{1}{q \alpha^y}.$$
For $x, y > 0$, one gets 
$$L_0(x,y) = \frac{q (\alpha^{x \wedge y}
- 1)}{(1-2q)},$$
and 
$$L_0(0,y) = 1.$$
Hence, for all $x \in \mathbb{N}_0$,  $L_0(x,y)$ converges when $y$ goes to infinity. The Martin boundary has then 
only one point denoted $\infty$, and 
$$L_0(x,\infty) = \frac{q (\alpha^{x}
- 1)}{(1-2q)}, L_0(0, \infty) = 1.$$
In this setting, the exists a unique 
stationnary probability measure, given 
by $$\beta(0) = \frac{1 - 2q}{2(1-q)}$$
and for all $x > 0$, 
$$\beta(x) = \frac{1 - 2q}{2q(1-q) 
\alpha^x}.$$
With this normalization, we get for all 
$x \geq 1$:
$$\varphi_{0, \infty}( x) = \frac{2q(1-q)}{(1-2q)^2} (\alpha^x - 1).$$
We then get, up to a multiplicative constant, a unique family of $\sigma$-finite measures $(\mathbb{Q}_x^{\infty})_{x \in \mathbb{N}_0}$, described in Subsection 4.3.2 of \cite{NRY} in the case $q = 1/3$, and then 
$\alpha = 2$. 
\subsection{The random walk on a tree}
Here, we consider an infinite $k$-ary tree 
for $k \geq 2$. It can be represented by 
the set $E$ of all finite (possibly empty)
sequences of elements in $\{0, 1, \dots, 
k-1\}$. The transition probabilities we consider are given by $p_{\emptyset, (j)}
= 1/k$ for all $j  \in \{0, \dots, k-1\}$, 
$p_{x,y} = 1/2$, $p_{y,x} =1/2k$ if $x$ is a nonempty sequence and $y$ is obtained 
from $x$ by removing the last element: we
will say that $y$ is the father of $x$ and $x$ is a son of $y$.  All the other transition probabilities are equal to zero. 
With these transitions, under $\mathbb{P}_x$ for any $x \in E$, $(L_n)_{n \geq 1}$ is a reflected standard random walk if 
$L_n$ denotes the length of the sequence $X_n$. One deduces that the Markov chain is irreducible and recurrent. 
We choose $x_0 = \emptyset$. 
Let $x, y \in E$, $x, y  \neq \emptyset$, and let $z$ be the last common ancestor of 
$x$ and $y$.
 It is clear that under 
$\mathbb{P}_x$, the canonical process
almost surely hits $z$ before $y$. Using the strong Markov property, we deduce that 
for $z = \emptyset$,  
$$G_{\emptyset}(x,y) = 0,$$
and for $z \neq \emptyset$, 
$$G_{\emptyset}(x,y) = G_{\emptyset} (z,y).$$
Let $z_0 = \emptyset, z_1, z_2, \dots, 
z_p = y$ be the ancestors of $y$, the sequence $z_j$ having $j$ elements. 
Under $\mathbb{P}_{z_j}$, $2 \leq j \leq p-1$, $X_1 = z_{j-1}$ with probability $1/2$, 
$X_1 = z_{j+1}$ with probability $1/2k$
and $X_1$ is another son of $z_j$ with probability $(k-1)/2k$. One deduces 
$$G_{\emptyset}(z_j, y) = \frac{1}{2} G_{\emptyset} (z_{j-1}, y) + 
\frac{1}{2k} G_{\emptyset} (z_{j+1}, y) + 
\frac{k-1}{2k} G_{\emptyset}(z_j,y),$$
and then 
$$G_{\emptyset} (z_j, y) = \frac{k}{k+1} G_{\emptyset} (z_{j-1}, y) + 
\frac{1}{k+1} G_{\emptyset} (z_{j+1}, y).$$
Similarly, for $p \geq 2$, 
$$G_{\emptyset} (z_1, y) 
= \frac{1}{k+1} G_{\emptyset} (z_{2}, y)$$
and 
$$G_{\emptyset} (y,y) = 
G_{\emptyset} (z_{p},y) =  
1 + \frac{1}{2} G_{\emptyset} (z_{p},y)
+ \frac{1}{2} G_{\emptyset} (z_{p-1},y),$$
i.e. 
$$G_{\emptyset} (y,y) = 2 + 
G_{\emptyset} (z_{p-1},y).$$
Finally, for $p = 1$, 
$$G_{\emptyset} (y,y) = 1 + 
\frac{1}{2} G_{\emptyset} (y,y) = 2.$$
From these equations, we deduce for $1 \leq j \leq p$, 
$$G_{\emptyset} (z_j,y) = 
\frac{2(k^j-1)}{k^{p-1} (k-1)}.$$
Hence, for $x, y \neq \emptyset$, 
$$G_{\emptyset} (x,y)  = 
\frac{2(k^j-1)}{k^{p-1} (k-1)},$$
where $j$ denotes the number of elements 
of the last common ancestor of $x$ and $y$. Moreover, one has $G_{\emptyset} (x, 
\emptyset) = 0$ for $x \neq \emptyset$, 
$G_{\emptyset} (\emptyset, 
\emptyset) = 1$, and for $y = z_p \neq \emptyset$, 
$$G_{\emptyset} (\emptyset, y) = 
\frac{1}{k} G_{\emptyset} (z_1, y)
= \frac{2}{k^p}.$$
We then get for $x, y \neq \emptyset$,  
$$L_{\emptyset} (x, y) 
= \frac{k (k^j -1)}{k-1},$$
$$L_{\emptyset} (x, \emptyset)  = 0, \;
L_{\emptyset} (\emptyset, y) = 
L_{\emptyset} (\emptyset, \emptyset)  = 1.
$$
Now, let $(x_n)_{n \geq 1}$ be a sequence 
in $E$ which converges in its Martin 
compactification. If the length of $x_n$ 
does not tend to infinity, then $x_n$ takes the same value infinitely often, 
and then the limit of $x_n$ is equal to 
this value. If the length of $x_n$ tends 
to infinity, then for all $m \geq 1$, 
$x_n$ has at least $m$ elements for $n$ large enough. Necessarily, there exists 
a sequence $y_m$ of $m$ elements such that 
$x_n$ starts with $y_m$ infinitely often. 
Now, let us assume that another sequence $y'_m$ satisfies the same property. 
In this case, the last common ancestor 
of $x_n$ and $y_m$ is $y_m$ itself infinitely often, and some strict ancestor 
of $y_m$ infinitely often. Hence, 
$L_{\emptyset}(x_n, y_m)$ is 
$k(k^m -1)/(k-1)$ infinitely often, 
and $k(k^j -1)/(k-1)$ for $j < m$ infinitely often, which contradicts the 
convergence of $(x_n)_{n \geq 1}$ in the 
Martin compactification of $E$. 
Hence, for all $m \geq 1$, there exists 
a sequence $y_m$ of length $m$ such that 
$x_n$ starts with $y_m$ for all $n$ large enough. 
Conversely, if $(x_n)_{n \geq 1}$ satisfies the property given in the previous sentence, then for all $x \in E$
different from $\emptyset$, 
and for $n$ large enough, 
the last common ancester of $x_n$ and $x$
is the same as the last common ancestor of $y_m$ and $x$, if $m \geq 1$ is larger than or equal to the length of $x$. Hence, for $n$ large enough, 
$$L_{\emptyset}(x, x_n) = 
L_{\emptyset}(x, y_m),$$
which implies the convergence of $(x_n)_{n \geq 1}$ in the Martin compactification of 
$E$. 
If $(x_n)_{n \geq 1}$ converges to a limit 
which is not in $E$, then the limit 
$\alpha$ in the Martin 
boundary of $E$ is defined by the function 
$$x \mapsto L_{\emptyset} (x, \alpha) 
= \lim_{n \rightarrow \infty}
L_{\emptyset} (x, x_n) = L_{\emptyset} 
(x, y_m),$$
if $m \geq 1$ is at least the length of $x$. The sequences $(y_m)_{m \geq 1}$ are compatible with each other, i.e. for all 
$m \geq 1$, $y_m$ is the father of $y_{m+1}$. Hence, there exists an sequence 
$y_{\infty}$ such that for all $m \geq 1$, 
$y_m$ is the sequence of the $m$ first elements of $y_{\infty}$. 
We deduce that $\alpha \in \partial E$ is 
identified by the function: 
$$ x \mapsto L_{\emptyset} (x, \alpha)
=  \mathds{1}_{x = \emptyset} + \frac{k (k^j - 1)}{k-1} \mathds{1}_{x
 \neq \emptyset} ,$$
where $j$ is the largest integer such that 
the $j$ first elements of $x$ and $y_{\infty}$ are the same. 
The function 
 $x \mapsto L_{\emptyset} (x, \alpha)$
depends only on the sequence $y_{\infty}$, 
and one easily checks that we obtain 
different functions for different infinite sequences. Hence $\alpha$ can be identified with $y_{\infty}$. 
We deduce that $\partial E$ can be identified with the set of all infinite sequences of elements in $\{0, 1, \dots, k-1\}$. Intuitively, these sequences correspond to the "leafs of the tree". 
The topology of $\partial E$ is given by 
the following convergence: $(y^{(n)}_{\infty})_{n \geq 1}$ converges to $y_{\infty}$ if and only if for all $k \geq 1$, the $k$-the element of
$y^{(n)}_\infty$ and the $k$-th element of 
$y_{\infty}$ coincide for all but finitely many $n \geq 1$.  The topology of $\partial E$ is also induced by the  
distance $D$ such that $D(y_{\infty}, 
y'_{\infty}) = 2^{-k}$, where $k$ is the 
first integer such that the $k$-th elements of $y_{\infty}$ and $y'_{\infty}$ are different.  
If we normalize the stationnary measure in such a way that $\beta(\emptyset) = 
k/(k-1)$, we get 
$$\varphi_{\emptyset, y_{\infty}} 
(x) = k^{j} -1$$ 
where $j$ is the largest integer such that 
the $j$ first elements of $y_{\infty}$ and 
$x$ coincide. This function is nonnegative, zero at $\emptyset$ and harmonic at every other point. It is 
minimal among such functions. Indeed, 
let $\varphi$ be a function satisfying the same properties, which is smaller than or equal 
to $\varphi_{\emptyset, y_{\infty}}$. 
For $m \geq 0$, let $E_m$ be the subset 
of $E$ consisting of sequences whose $m$ 
first elements coincide with those of $y_{\infty}$. Let $F_m := E_m \backslash 
E_{m+1}$. If $y_m$ is the sequence of the 
$m$ first elements of $y_{\infty}$, then 
for $x \in F_m$, under $\mathbb{P}_x$, 
$X_n$ is in $F_m \backslash \{y_m\}$ for all $n < T_{y_m}$. 
Now, for all $z \in F_m$, 
$$\varphi(z) \leq \varphi_{\emptyset, y_{\infty}} (z) = k^m - 1,$$
which implies, since $\varphi$ is harmonic 
at all points in $F_m \backslash \{y_m\}$, that $(\varphi(X_{n \wedge T_{y_m}}))_{n \geq 0}$ is a bounded martingale. Hence, 
under $\mathbb{P}_x$, 
$$\varphi(x) =  \varphi(X_0) 
= \mathbb{E}_x [ \varphi(X_{T_{y_m}})]
= \varphi(y_m).$$
Hence, $\varphi(x)$ depends only on 
the largest $m$ such that the $m$ first elements of $x$ and $y_{\infty}$ coincide. 
It is then sufficient to check that there
exists $C > 0$ such that  
$\varphi(y_m) = C(k^m - 1)$ for all $m \geq 0$. Now, this result is a consequence of the fact that $\varphi(y_0) = 
\varphi(\emptyset) = 0$ and the 
harmonic property of $\varphi$, which imply that for $m \geq 1$, 
$$\varphi(y_m) = \frac{1}{2}
\varphi(y_{m-1}) + \frac{k-1}{2k} 
\varphi(y_m)  + \frac{1}{2k} \varphi(y_{m
+1}).$$
We have then proven that $\varphi_{\emptyset, y_{\infty}}$ is minimal, and then the minimal boundary 
$\partial_m E$ 
is equal to the Martin boundary $\partial E$. 
We have then again a complete description 
of the families of measures $(\mathbb{Q}^{\alpha}_{x})_{x \in E}$ for $\alpha \in \partial_m E$. In the case $k=2$, all these measures were already describd in 
Subsection 4.3.3 of \cite{NRY}.
 Integrating with respect to 
a finite measure on $\partial_m E$ gives 
all the possible families of measures 
in the class $\mathcal{Q}$. Note that in this example, the Martin boundary is uncountable. 

\subsection*{Acknowledgment}
I would like to thank Ph. Biane and Ph. Bougerol for the discussion we had on the possibility of a link between the Martin  
boudary and the $\sigma$-finite measures studied here.

\providecommand{\bysame}{\leavevmode\hbox to3em{\hrulefill}\thinspace}
\providecommand{\MR}{\relax\ifhmode\unskip\space\fi MR }
\providecommand{\MRhref}[2]{%
  \href{http://www.ams.org/mathscinet-getitem?mr=#1}{#2}
}
\providecommand{\href}[2]{#2}

\end{document}